\documentclass[a4paper,twoside,11pt]{article}
\usepackage{a4wide, fancyhdr, amsmath, amssymb, mathtools, yfonts}
\usepackage{graphicx}
\usepackage{tikz}
\usepackage[all]{xy}
\usepackage[utf8]{inputenc}
\usepackage{amsthm}
\usepackage[english]{babel}
\usepackage{chngcntr}
\usepackage{ifthen}
\usepackage{calc}
\usepackage{hyperref}
\usepackage[OT2,T1]{fontenc}
\usepackage{authblk}
\numberwithin{equation}{section}
\DeclareSymbolFont{cyrletters}{OT2}{wncyr}{m}{n}
\DeclareMathSymbol{\Sha}{\mathalpha}{cyrletters}{"58}

\providecommand{\keywords}[1]{{2010}\textit{ Mathematics Subject Classification. }#1}

\newcommand{\Z}{\mathbb{Z}}
\newcommand{\Q}{\mathbb{Q}}
\newcommand{\QP}{\mathbb{Q}(\sqrt{-p})}
\newcommand{\QTP}{\mathbb{Q}(\sqrt{2p})}

\newcommand{\OO}{\mathcal{O}}
\newcommand{\CL}{\mathrm{Cl}}
\newcommand{\C}{\mathrm{Cl}^{+}}

\newcommand\FF{\mathbb{F}}
\newcommand\MM{\mathbb{M}}
\newcommand\DD{\mathcal{D}}

\newcommand\RRR{\mathcal{R}}

\newcommand\aaa{\mathfrak{a}}

\newcommand\bb{\mathfrak{b}}

\newcommand\mm{\mathfrak{m}}

\newcommand\pp{\mathfrak{p}}

\newcommand\pt{\mathfrak{t}}

\newcommand\rk{\mathrm{rk}}

\newcommand\Gal{\mathrm{Gal}}
\newcommand\Norm{\mathrm{N}}

\newcommand\Disc{\mathrm{Disc}}

\newcommand\Art{\mathrm{Art}}

\newcommand\spin{\mathrm{spin}}

\newcommand\ve{\varepsilon}

\newtheorem{theorem}{Theorem}
\newtheorem{lemma}{Lemma}[section]
\newtheorem{prop}[lemma]{Proposition}
\newtheorem{corollary}[theorem]{Corollary}

\newtheorem{conjecture}{Conjecture}

\title{Spins of prime ideals and the negative Pell equation $x^2 - 2py^2 = -1$}
\author[1]{Peter Koymans\thanks{Niels Bohrweg 1, 2333 CA Leiden, Netherlands, p.h.koymans@math.leidenuniv.nl}}
\author[2]{Djordjo Milovic\thanks{Gower Street, London, WC1E 6BT, United Kingdom, djordjo.milovic@ucl.ac.uk}}
\affil[1]{Mathematisch Instituut, Leiden University}
\affil[2]{Department of Mathematics, University College London}

\date{\today}

\begin{document}

\maketitle

\begin{abstract}
Let $p\equiv 1\bmod 4$ be a prime number. We use a number field variant of Vinogradov's method to prove density results about the following four arithmetic invariants: (i)~$16$-rank of the class group $\CL(-4p)$ of the imaginary quadratic number field $\Q(\sqrt{-4p})$; (ii)~$8$-rank of the ordinary class group $\CL(8p)$ of the real quadratic field $\Q(\sqrt{8p})$; (iii)~the solvability of the negative Pell equation $x^2 - 2py^2 = -1$ over the integers; (iv)~$2$-part of the Tate-\v{S}afarevi\v{c} group $\Sha(E_p)$ of the congruent number elliptic curve $E_p: y^2 = x^3-p^2x$. Our results are conditional on a standard conjecture about short character sums. 
\end{abstract}

\section{Introduction}\label{Introduction}
In \cite{FIMR}, Friedlander, Iwaniec, Mazur, and Rubin associated a quantity $\spin(\aaa)\in\{0, \pm 1\}$ to each principal ideal $\aaa$ in the ring of integers of a totally real number field $K$ of degree $n\geq 3$ with a \textit{cyclic} Galois group over $\Q$. Assuming a standard conjecture about short character sums, they proved that $\spin(\pp)$ oscillates as $\pp$ varies over principal prime ideals. The conjecture is unconditional in the low-degree case when $n = 3$, and precisely in this setting their result has arithmetic applications to the distribution of 2-Selmer groups of quadratic twists of certain elliptic curves.

In this paper, we will associate a similar ``spin'' to ideals in the ring of integers $\OO_M$ of the totally complex number field
$$
M = \Q(\zeta_8, \sqrt{1+i}),
$$
where $\zeta_8$ is a primitive $8$th root of unity and $i = \zeta_8^2$. The essential part of our spin will come from symbols of the type 
\begin{equation}\label{defTotSpin}
[\alpha]_r = \left(\frac{r(\alpha)}{\alpha}\right),
\end{equation}
where $\left(\frac{\cdot}{\cdot}\right)$ is the quadratic residue symbol in $M$ and $r\in\Gal(M/\Q)$ is a fixed automorphism of order $4$. Following the basic strategy of \cite{FIMR}, we will also prove that the spin of prime ideals in $\OO_M$ oscillates. Unfortunately, the field $M$ is of degree $8$ over $\Q$, and we are forced to assume the $n = 8$ case of \cite[Conjecture $C_n$, p.\ 738]{FIMR}. Our result has applications to the arithmetic statistics of: (i) the $16$-rank of the class group of $\QP$, (ii) the $8$-rank of the \textit{ordinary} class group of the \textit{real} quadratic field $\QTP$, (iii) the negative Pell equation $x^2 - 2py^2 = -1$, and (iv) the congruent number elliptic curve $y^2 = x^3 - p^2x$.

There are two main innovations that separate the present work from \cite{FIMR}. First, we have a multitude of new arithmetic applications, made possible by carefully crafting a flexible new type of spin. Secondly, the Galois group of $M/\Q$ is dihedral of order~$8$, hence not cyclic, and this seemingly technical difference causes the original arguments in \cite{FIMR} to break down. Fortunately a lattice point counting argument offers a fix which also substantially simplifies the proof in \cite{FIMR}.   

Before stating our main results, we define the aforementioned spin $s_{\aaa}$ of non-zero ideals $\aaa\subset\OO_M$. One can check that $M/\Q$ is a totally complex dihedral extension of degree $8$, that $\OO_M$ is a principal ideal domain, and that $\zeta_8$ generates the torsion subgroup of the unit group~$\OO_M^{\times}$. We fix a subgroup $V\leq \OO_M^{\times}$ of rank $3$ such that $\OO_M^{\times} = \left\langle \zeta_8\right\rangle \times V$ and fix a set of coset representatives $\mu_1, \cdots, \mu_8$ for $V^2$ in $V$. We define a rational integer $F$ as in \eqref{defF}; although $F$ is an absolute constant, it is far too large to write out its decimal expansion. Suppose that
\begin{equation}\label{veModF}
\psi: (\OO_M/F\OO_M)^{\times}\rightarrow \mathbb{C}
\end{equation}
is a map such that $\psi(\alpha \bmod F) = \psi(\alpha\beta^2 \bmod F)$ for all $\alpha\in\OO_M$ coprime to $F$ and all $\beta\in\OO_M^{\times}$. Fix once and for all an element of order $4$ in $\Gal(M/\Q)$, denote it by $r$, and define~$[\cdot]_r$ as in \eqref{defTotSpin}. Finally, let $\aaa$ be a non-zero ideal in $\OO_M$. If $(\aaa, F)\neq 1$, define $s_{\aaa} = 0$. Otherwise, choose any generator $\alpha$ for $\aaa$ and define
\begin{equation}\label{defan}
s_{\aaa} = \frac{1}{64}\sum_{i = 1}^8\sum_{j = 1}^8 \psi(\mu_i \zeta_8^j \alpha\bmod F)\cdot [\mu_i \zeta_8^j \alpha]_r.
\end{equation}
The right hand side above is independent of the choice of a generator $\alpha$ for $\aaa$, as can be seen from \eqref{welldef} with $\sigma = r$. Compare the definition of $s_{\aaa}$ with the definition of $\spin(\aaa)$ in \cite[(3.4), p.\ 706]{FIMR}. The most important difference is that $r$ does \textit{not} generate the Galois group $\Gal(M/\Q)$, whereas in \cite{FIMR}, the automorphism $\sigma$ does generate $\Gal(K/\Q)$. This requires us to make an innovation in perhaps the most difficult part of the analytic arguments of Friedlander et al.\ \cite[p.\ 731-733]{FIMR}. Another difference is the extra averaging over generators of $\aaa$ in the definition of $s_{\aaa}$ above, necessary because, unlike in \cite{FIMR}, we cannot make simplifying assumptions about the field over which we work.

We now state our main theorem and its consequences, all conditional on Conjecture~\ref{conjSCS}, a standard conjecture about short character sums whose statement we postpone to Section~\ref{ConjectureSection}.

\begin{theorem}\label{mainThm}
Assume that Conjecture~\ref{conjSCS} holds with $\delta>0$. Then there is a constant $\delta' > 0$ depending only on $\delta$ such that for all $X>1$, we have
$$
\sum_{\Norm(\pp)\leq X}s_{\pp}\ll X^{1-\delta'},
$$
where the sum is taken over prime ideals $\pp\subset \OO_M$ of norm at most $X$ and the implied constant depends only on $\psi$. Moreover, one can take $\delta' = \delta/400$.
\end{theorem}

Let $\CL(D)$, $\C(D)$, $h(D)$, and $h^+(D)$ denote the class group, the narrow class group, the class number, and the narrow class number, respectively, of the quadratic field of discriminant~$D$. For a finite abelian group $G$ and an integer $k\geq 1$, we define the $2^k$-rank of~$G$ to be $\rk_{2^k}G =  \dim_{\FF_2}(2^{k-1}G/2^kG)$. A lot is known about the $8$-rank of $\C(dp)$ for $d$ fixed and $p$ varying among the prime numbers (see \cite{Ste1} and \cite{Smith1}). We will prove some long-standing conjectures about the $16$-rank of $\CL(-4p)$ and the $8$-rank of $\CL(8p)$ (see for instance \cite{CohnLag2} and in particular their Density Conjecture $D_j(d)$ on page 263).

\begin{theorem}\label{Thmi}
Assume that Conjecture~\ref{conjSCS} holds with $\delta > 0$ and let $\delta'$ be as in Theorem~\ref{mainThm}. Let $r\in\{0, 8\}$. For all $X \geq 41$, we have
$$
\frac{\#\{p\leq X: h(-4p)\equiv r\bmod 16\}}{\#\{p\leq X: h(-4p)\equiv 0\bmod 8\}} = \frac{1}{2} + O(X^{-\delta'}),
$$
where the implied constant is absolute.
\end{theorem}
\begin{theorem}\label{Thmii}
Assume that Conjecture~\ref{conjSCS} holds with $\delta > 0$ and let $\delta'$ be as in Theorem~\ref{mainThm}. Let $r\in\{0, 4\}$. Then for all $X\geq 113$, we have
$$
\frac{\#\{p\leq X: p\equiv 1\bmod 4, h(8p)\equiv r\bmod 8\}}{\#\{p\leq X: p\equiv 1\bmod 4, h^+(8p)\equiv 0\bmod 8\}} = \frac{1}{2} + O(X^{-\delta'}),
$$
where the implied constant is absolute.
\end{theorem}

Density results about the $2$-parts of the narrow and ordinary class groups of $\Q(\sqrt{8p})$ have implications for the arithmetic statistics of the solvability of the negative Pell equation
\begin{equation}\label{nPell}
x^2 - 2py^2 = -1
\end{equation}
with $x, y\in \Z$. For each $X\geq 3$, let
$$
\delta^-(X) = \frac{\#\{p\text{ prime}: p\leq X, \eqref{nPell} \text{ is solvable over }\Z\}}{\#\{p\text{ prime}: p\leq X\}}.
$$
Stevenhagen conjectured in \cite{SteConj} that $\lim_{X\rightarrow\infty}\delta^-(X)$ exists and equals to $1/3$. We prove

\begin{theorem}\label{Thmiii}
Assume that Conjecture~\ref{conjSCS} holds. Let $\delta^-(X)$ be defined as above. Then
$$
\frac{5}{16}\leq \liminf_{X\rightarrow\infty}\delta^-(X) \leq \limsup_{X\rightarrow\infty}\delta^-(X) \leq \frac{11}{32}.
$$ 
In particular, $|\delta^{-}(X)-1/3| \leq 1/48 + o(X)$ as $X\rightarrow \infty$, so the bounds above are within $2.08\%$ of Stevenhagen's Conjecture.
\end{theorem}

Finally, we state an application of Theorem~\ref{mainThm} to the distribution of the Tate-\v{S}afarevi\v{c} groups $\Sha(E_p)$ of the congruent number elliptic curves
$$
E_p: y^2 = x^3-p^2x.
$$

\begin{theorem}\label{Thmiv}
Assume that Conjecture~\ref{conjSCS} holds. Then
$$
\liminf_{X\rightarrow\infty}\frac{\#\{p\leq X:\ \left(\Z/4\Z\right)^2\hookrightarrow\Sha(E_p)\}}{\#\{p\leq X\}}\geq \frac{1}{16}.
$$
\end{theorem}

\section{Discussion of results}

\subsection{$16$-rank of class groups}
Aside from two recent results due to the authors \cite{Milovic2, KM2}, density results about the $16$-rank of class groups in one-prime-parameter families $\{\Q(\sqrt{dp})\}_p$ ($d$ fixed and $p$ varying) have remained elusive despite a large body of work on algebraic criteria for the $16$-rank in such families \cite{Kaplan77, Oriat78, KW82, LW82, KW84, Y84, KWH86, Ste2, BH}. This gap between algebraic and analytic understanding of the $16$-rank can be largely attributed to the absence of appropriate governing fields and the subsequent inability to apply the \v{C}ebotarev Density Theorem. More precisely, for a finite extension of number fields $E/F$, let $\Art_{E/F}$ denote the corresponding Artin map. Cohn and Lagarias~\cite{CohnLag, CohnLag2} conjectured that, for each integer $k\geq 1$ and each integer $d\not\equiv 2\bmod 4$, the map
$$
f_{d, k}: p\mapsto \rk_{2^k}\C(dp)
$$
is \textit{Frobenian}, in the sense of Serre \cite{JPS2}. In other words, they conjectured that there exists a normal field extension $M_{d, k}/\Q$ for which there is a class function
$$
\phi: \Gal(M_{d, k}/\Q)\rightarrow \Z_{\geq 0}
$$
satisfying
$$
f_{d, k}(p)= \phi(\Art_{M_{d, k}/\Q}(p))
$$
for all primes $p$ unramified in $M_{d, k}/\Q$; such a field $M_{d, k}$ is called a \textit{governing field} for $\{\rk_{2^k}\C(dp)\}_p$. For $k\leq 3$, Stevenhagen \cite{Ste1} proved these conjectures for all $d\not\equiv 2\bmod 4$. Perhaps the simplest case is $d = -4$, where one can take $M_{-4, 3}$ to be the field $M = \Q(\zeta_8, \sqrt{1+i})$ as above and where $h(-4p) \equiv 0\bmod 8$ if and only if $p$ splits completely in $M$. Hence, by the \v{C}ebotarev Density Theorem, the density of primes $p$ such that $h(-4p)\equiv 0\bmod 8$ is equal to $1/[M:\Q] = 1/8$.

Cohn and Lagarias~\cite{CohnLag2} ruled out some obvious candidates for $M_{-4, 4}$, i.e., the governing field for the $16$-rank of $\CL(-4p)$, and to this day no governing fields for the $16$-rank in \textit{any} family have been found. Nevertheless, we're able to show, in Theorem~\ref{Thmi}, that the density of primes $p$ such that $h(-4p)\equiv 0\bmod 16$ exists and is equal to $1/16$. It is proved unconditionally in \cite{Milovic1} that there are infinitely many primes~$p$ such that $h(-4p)\equiv 0\bmod 16$, but that result implies nothing about the density as in Theorem~\ref{mainThm}.

The key innovation that allows us to go beyond the $8$-rank is to use Vinogradov's method \cite{Vino1, Vino2} for studying the distribution of prime numbers instead of the heretofore used \v{C}ebotarev Density Theorem (as in \cite{Smith1}, for instance). Moreover, the current state-of-the-art bounds for the error term in the \v{C}ebotarev Density Theorem are essentially of size $X\exp(-\sqrt{\log X})$, far worse than the power-saving bound $X^{1-\delta'}$ in Theorem~\ref{Thmi}. In fact, obtaining such a power-saving error term in the \v{C}ebotarev Density Theorem would be tantamount to proving a zero-free region for the associated Artin $L$-functions of the form $\Re(s)> 1- \delta'$, and this is well out of reach of current methods in analytic number theory. Nonetheless, the power-saving bound $X^{1-\delta'}$ does \textit{not} prove the non-existence of a governing field -- it merely suggests that one is unlikely to exist. We summarize this discussion with the following immediate corollary of Theorem~\ref{Thmi}. 
\begin{corollary}
Assume Conjecture~\ref{conjSCS} with $\delta>0$, and let $\delta'$ be as in Theorem~\ref{mainThm}. At least one of the following two statements is true:
\begin{itemize}
\item a governing field for $\rk_{16}\CL(-4p)$ does \textit{not} exist;
\item there exists a normal extension $L/\Q$ and two distinct unions of conjugacy classes in $\Gal(L/\Q)$, say $S_1$ and $S_2$, such that for all $X>0$, we have
$$
\#\{p\leq X: (p, L/\Q)\subset S_1\} - \#\{p\leq X: (p, L/\Q)\subset S_2\} \ll X^{1-\delta'},
$$
where the implied constant is absolute. Here $(p, L/\Q)$ denotes the Artin conjugacy class of $p$ in $\Gal(L/\Q)$.
\end{itemize}
\end{corollary}

\subsection{Real quadratic fields and the negative Pell equation}
In case $d<0$, the narrow class group $\C(dp)$ is the same as the ordinary class group $\CL(dp)$. If $d>0$, however, then $\C(dp)$ and $\CL(dp)$ may be different; in fact, $\C(dp) = \CL(dp)$ if and only if the fundamental unit $\ve_{dp}$ of $\Q(\sqrt{dp})$ has norm $-1$. While Cohn and Lagarias stated their conjecture on the existence of governing fields only for narrow class groups, one can ask what happens for ordinary class groups. As mentioned before, Stevenhagen proved the conjecture of Cohn and Lagarias for the $8$-rank of narrow class groups of both imaginary and real quadratic fields. Theorem~\ref{Thmii} is the first density result for the $8$-rank of the \textit{ordinary} class group in a family of \textit{real} quadratic fields. Again the power-saving error term suggests that there is no governing field for $\rk_8\CL(8p)$ in the family $\{\CL(8p)\}_{p\equiv 1\bmod 4}$. To place Theorem~\ref{Thmii} in context, we note that the $2$-part of $\C(8p)$ is cyclic, and, for $p \equiv 1\bmod 4$, one has (for instance, see \cite{Ste2})
\begin{itemize}
\item $h^+(8p) = h(8p) \equiv 2\bmod 4 \Leftrightarrow p \text{ splits completely in }\Q(i)\text{ but not in }\Q(\zeta_{8})$;
\item $h^+(8p) \equiv h(8p) + 2 \equiv 0\bmod 4 \Leftrightarrow p \text{ splits completely in }\Q(\zeta_8)\text{ but not in }\Q(\zeta_{8}, \sqrt[4]{2})$;
\item $h^+(8p) = h(8p) \equiv 4\bmod 8 \Leftrightarrow p \text{ splits completely in }\Q(\zeta_8, \sqrt[4]{2})\text{ but not in }\Q(\zeta_{16}, \sqrt[4]{2})$;
\item $h^+(8p)\equiv 0\bmod 8 \Leftrightarrow p \text{ splits completely in }\Q(\zeta_{16}, \sqrt[4]{2})$.
\end{itemize}
Hence, Theorem~\ref{Thmii} in conjunction with the \v{C}ebotarev Density Theorem implies that
$$
\lim_{X\rightarrow \infty}\frac{\#\{p \text{ prime}: p\leq X, p\equiv 1\bmod 4, h(8p)\equiv 4\bmod 8\}}{\#\{p \text{ prime}: p\leq X, p\equiv 1\bmod 4\}} = \frac{3}{16}
$$
and 
$$
\lim_{X\rightarrow \infty}\frac{\#\{p \text{ prime}: p\leq X, p\equiv 1\bmod 4, h(8p)\equiv 0\bmod 8\}}{\#\{p \text{ prime}: p\leq X, p\equiv 1\bmod 4\}} = \frac{1}{16}.
$$

The 2-torsion subgroup $\C(8p)[2]$ is generated by the classes of the ramified ideals $\pt$ and $\pp$ lying above $2$ and $p$, respectively. Since the $2$-part of $\C(8p)$ is cyclic, we have $\#\C(8p)[2] = 2$, so exactly one of the three ideals $\pt$, $\pp$, and $\pt\pp$ is in the trivial class in $\C(8p)$, while the remaining two are both in the non-trivial class in $\C(8p)[2]$. Moreover, \eqref{nPell} has a solution over the integers if and only if $\Z[\sqrt{2p}]$ has a unit of norm $-1$, which occurs if and only if the ideal $\pt\pp = (\sqrt{2p})$ can be generated by a totally positive element in $\Z[\sqrt{2p}]$, i.e., if and only if $\pt\pp$ is in the trivial class in $\C(8p)$. Stevenhagen conjectured in \cite{SteConj} that as $p$ varies over all prime numbers, each of $\pt$, $\pp$, and $\pt\pp$ is in the trivial class in $\C(8p)$ equally often, which is why we expect $\lim_{X\rightarrow \infty}\delta^-(X)$ to exist and equal to $1/3$ ($\delta^-(X)$ is defined following \eqref{nPell}).

Since $\Z[\sqrt{2p}]$ has a unit of norm $-1$ if and only if the narrow class group $\C(8p)$ coincides with the ordinary class group $\CL(8p)$, we can obtain successively better upper and lower bounds for the proportion of primes $p$ for which \eqref{nPell} is solvable over $\Z$ by comparing $h^+(8p)$ and $h(8p)$ modulo successively higher powers of $2$. Note that \eqref{nPell} has no solutions (even over $\Q$) whenever $p\equiv 3\bmod 4$, since in that case $-1$ is not a quadratic residue modulo $p$. From this, the list of splitting criteria above, and the \v{C}ebotarev Density Theorem, one immediately deduces that
\begin{equation}\label{oldbounds}
\frac{5}{16}\leq \liminf_{X\rightarrow\infty}\delta^-(X) \leq \limsup_{X\rightarrow\infty}\delta^-(X) \leq \frac{3}{8}.
\end{equation}
Hence $|\delta^{-}(X) - 1/3| \leq 1/24 + o(X)$ as $X\rightarrow \infty$, i.e., at worst, the bounds above are within $4.17\%$ of Stevevenhagen's Conjecture. Theorem~\ref{Thmiii} hence cuts the possible discrepancy from Stevenhagen's conjecture in half. Although the problem of improving \eqref{oldbounds} may have been first explicitly stated in 1993 in \cite[p.\ 127]{SteConj}, in essence it has been open since the 1930's, when R\'{e}dei~\cite{Redei}, Reichardt~\cite{Reichardt}, and Scholz~\cite{Scholz} supplied the algebraic criteria sufficient to deduce \eqref{oldbounds}.

\subsection{Other results on $2$-parts of class groups of number fields}
Finally, we would like to contrast our results concerning one-prime-parameter families with results on $2$-parts of class groups in families parametrized by arbitrarily many primes. The first significant achievement for families with arbitrary discriminants was made by Fouvry and Kl\"{u}ners \cite{FK2}, who translated R\'{e}dei's theory on $4$-ranks of class groups to sums of characters conducive to analytic techniques and then successfully dealt with these sums, basing some of their work on the techniques developed by Heath-Brown in \cite{HB1, HB2}. Fouvry and Kl\"{u}ners subsequently developed their methods in various settings \cite{FK3, FK4, FK5, FK6}, most notably obtaining impressive upper and lower bounds for the solvability of the negative Pell equation $x^2 - dy^2 = -1$ for general squarefree integers $d>0$. When specialized to the one-prime-parameter family $d = 2p$ with $p$ prime, their results are as strong as the bounds in \eqref{oldbounds}, so Theorem~\ref{Thmiii} can be viewed as the next natural step in the line of work initiated by Fouvry and Kl\"{u}ners.

A recent paper of Smith~\cite{Smith2} (see also \cite{Smith1}) features ground-breaking distribution theorems about $2^k$-ranks of class groups of imaginary quadratic fields for \textit{all} $k\geq 3$. The very deep methods that underlie these theorems require the number of prime parameters on average to go to infinity and hence are unlikely to yield results in the direction of Theorems~\ref{Thmi}, \ref{Thmii}, or \ref{Thmiii}; from the standpoint of analytic number theory, Theorem~\ref{Thmi} is a result about the distribution of prime numbers, while the main analytic techniques underlying the results of \cite{Smith2} are consequences of a very careful study of the anatomy of the prime divisors of highly composite integers.

\subsection*{Acknowledgments}
The authors would like to thank Jan-Hendrik Evertse, \'{E}tienne Fouvry, Zev Klagsbrun, Carlo Pagano, and Peter Stevenhagen for useful discussions related to this work. The first author is a doctoral student at Leiden University. The second author was supported by an ALGANT Erasmus Mundus Scholarship, National Science Foundation agreement No.\ DMS-1128155, and European Research Council grant agreement No.\ 670239.

\section{Preliminaries}

\subsection{The governing field for the $8$-rank of $\CL(-4p)$}\label{GovField}
As in Section~\ref{Introduction}, let $M = \Q(\zeta_8, \sqrt{1+i})$ be the (minimal) governing field for the $8$-rank in the family $\{\Q(\sqrt{-4p})\}_{p\equiv 1\bmod 4}$. Using a computer algebra package such as Sage, one can readily check that
\begin{enumerate}
	\item[(P1)] the ring of integers of every subfield of $M$ (including $M$ itself) is a principal ideal domain,
	\item[(P2)] the discriminant $\Delta_M$ of $M/\Q$ is equal to $2^{22}$, and $2$ is totally ramified in $M/\Q$, and
	\item[(P3)] the torsion subgroup of the group of units in $\OO_M$ is $\left\langle \zeta_8 \right\rangle$.
\end{enumerate}
Recall that $\rk_8\CL(-4p) = 1$ if and only if $p$ splits completely in $M/\Q$, that is, if and only if $p$ is odd and every prime ideal $\pp$ in $\OO_M$ lying over $p$ is of degree $1$.

As noted in Section~\ref{Introduction}, $M/\Q$ is a normal extension with Galois group isomorphic to the dihedral group $D_8$ of order $8$. We fix an automorphism $r\in \Gal(M/\Q)$ such that $r$ generates the order $4$ subgroup $\Gal(M/\Q(\sqrt{-2}))$, and we let $s\in\Gal(M/\Q)$ be the non-trivial automorphism fixing the subfield $K_1 = \Q(i, \sqrt{1+i})$. Then $D_8 \cong\Gal(M/\Q)\cong \left\langle r, s\right\rangle$, with $r$ of order $4$, $s$ of order $2$, and $sr = r^3s$. Hereinafter, we refer to the following field diagram.
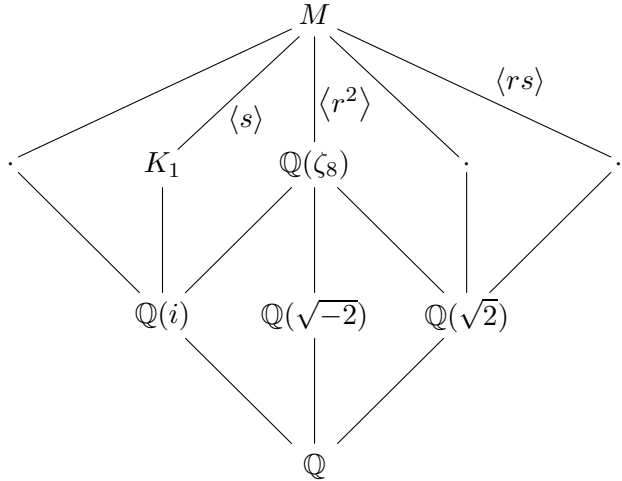
\begin{center}
\begin{tikzpicture}
  \draw (0, 0) node[]{$\Q$};
  \draw (0, 2) node[]{$\Q(\sqrt{-2})$};
	\draw (-2, 2) node[]{$\Q(i)$};
	\draw (2, 2) node[]{$\Q(\sqrt{2})$};
  \draw (0, 4) node[]{$\Q(\zeta_8)$};
	\draw (0, 6) node[]{$M$};
	\draw (-2, 4) node[]{$K_1$};
	\draw (-4, 4) node[]{$\cdot$};
	\draw (2, 4) node[]{$\cdot$};
	\draw (4, 4) node[]{$\cdot$};
	\draw (0, 0.3) -- (0, 1.7);
	\draw (0, 2.3) -- (0, 3.7);
	\draw (0, 4.3) -- (0, 5.7);
	\draw (0.3, 0.3) -- (1.7, 1.7);
	\draw (-0.3, 0.3) -- (-1.7, 1.7);
	\draw (1.7, 2.3) -- (0.3, 3.7);
	\draw (-1.7, 2.3) -- (-0.3, 3.7);
	\draw (2, 2.3) -- (2, 3.9);
	\draw (2.3, 2.3) -- (3.9, 3.9);
	\draw (1.9, 4.1) -- (0.2, 5.7);
	\draw (3.9, 4.1) -- (0.3, 5.8);
	\draw (-2, 2.3) -- (-2, 3.7);
	\draw (-2.3, 2.3) -- (-3.9, 3.9);
	\draw (-1.8, 4.2) -- (-0.2, 5.7);
	\draw (-3.9, 4.1) -- (-0.3, 5.8);
	\draw (0.4, 4.8) node[]{$\left\langle r^2\right\rangle$};
	\draw (-0.9, 4.6) node[]{$\left\langle s\right\rangle$};
	\draw (2.7, 5.1) node[]{$\left\langle rs\right\rangle$};
\end{tikzpicture}
\end{center} 
By the \v{C}ebotar\"{e}v Density Theorem, for each $\rho \in (\OO_M/(\Delta_M))^\times$, we can choose an inverse $\rho' \in \OO_M$ such that $\rho'\OO_M$ is a prime of degree one. Fix a set of such $\rho'$ and call it $\RRR$. Define $F$ to be the rational integer
\begin{equation}\label{defF}
F = \Delta_M\cdot \prod_{\rho \in (\OO_M/(\Delta_M))^\times}\Norm_{M/\Q}(\rho').
\end{equation}
This is not really analogous to $F$ on \cite[p.\ 723]{FIMR}, but we denote it by the same letter because it will play an analogous role later on in the estimation of certain congruence sums. 

\subsection{Quadratic Reciprocity}\label{sQR}
Let $L$ be a number field and let $\OO_L$ be its ring of integers. We say that an ideal $\aaa$ in $\OO_L$ is \textit{odd} if $\Norm(\aaa)$ is odd; similarly, an element $\alpha$ in $\OO_L$ is called \textit{odd} if the principal ideal generated by $\alpha$ is odd. If $\pp$ is an odd prime ideal in $\OO_L$, and $\alpha$ is an element in $\OO_L$, then one defines
$$
\left(\frac{\alpha}{\pp}\right)_{L} =
\begin{cases}
0 & \text{if }\alpha\in\pp \\
1 & \text{if }\alpha\notin\pp\text{ and }\alpha\text{ is a square modulo }\pp \\
-1 & \text{otherwise.}
\end{cases}
$$
If $\bb$ is an odd ideal in $\OO_L$, one defines
$$
\left(\frac{\alpha}{\bb}\right)_{L} = \prod_{\pp^{k_{\pp}}\|\bb}\left(\frac{\alpha}{\pp}\right)_{L}^{k_{\pp}}.
$$
If $\alpha, \beta\in\OO_L$ with $\beta$ odd, we define 
$$
\left(\frac{\alpha}{\beta}\right)_L = \left(\frac{\alpha}{\beta\OO_L}\right)_L.
$$
A weak (but sufficient to us) version of the law of quadratic reciprocity for number fields can be stated as follows (see for instance \cite[Lemma 2.1, p.\ 703]{FIMR}).
\begin{lemma}\label{tQR}
Suppose $L$ is a totally complex number field, and let $\alpha, \beta \in \OO_L$ be odd. Then
$$
\left(\frac{\alpha}{\beta}\right)_L = \ve \cdot \left(\frac{\beta}{\alpha}\right)_L,
$$
where $\ve\in\{\pm 1\}$ depends only on the congruence classes of $\alpha$ and $\beta$ modulo $8\OO_L$. $\Box$
\end{lemma}
When $\alpha$ is not odd, the following supplement to the law of quadratic reciprocity will suffice for our purposes (see \cite[Proposition 2.2, p.\ 703]{FIMR}).
\begin{lemma}\label{mod8}
Let $L$ be a totally complex number field, and let $\alpha \in \OO_L$ be non-zero. Then $\left(\frac{\alpha}{\beta}\right)_L$ depends only on the congruence class of $\beta$ modulo $8\alpha\OO_L$. $\Box$
\end{lemma}

\subsection{Short character sums}\label{ConjectureSection}
Here we state the conjecture that we assume in the proof of Theorem~\ref{mainThm}. It stipulates power-savings in short character (modulo $q$) sums of length $q^{1/8}$ and is essentially the same as the case $n = 8$ of Conjecture~$C_n$ in \cite[p.\ 738]{FIMR}.
\begin{conjecture}\label{conjSCS}
There exist absolute constants $\delta>0$ and $C>0$ such that if $\chi$ is a non-principal real-valued Dirichlet character modulo a squarefree integer $q>2$ and $N<q^{1/8}$, then 
$$
\left| \sum_{M\leq n \leq M+N}{\chi(n)} \right| \leq C q^{\frac{1}{8}-\delta}
$$
for all integers $M$.
\end{conjecture}
We feel that Conjecture~\ref{conjSCS} is of a genuinely different nature than the arithmetic applications that follow. It is the oscillation of spins over the set of \textit{prime} ideals that yields the various arithmetic applications. In the sieving methods we use, proving oscillation of spins over prime ideals requires us to first prove oscillation over the set of \textit{all} ideals. There we encounter character sums in the number field $M$ that one wishes to relate to character sums in $\Q$, where oscillation of character sums is better understood. In passing from $M$ to $\Q$, one suffers from the fact that, in some fixed integral basis for $\OO_M$, a nicely chosen element of norm $X$ generally has coordinates of size $X^{1/8}$. Conductors of characters in question have size similar to the norm, while the length of character sums in question is essentially limited by the size of the coordinates. We also remark that thanks to the work of Burgess~\cite{Burgess1, Burgess2}, Conjecture~\ref{conjSCS} is known to be true when $1/8$ is replaced with any real number $\theta>1/4$, in which case the exponent $\delta$ and the constant $C$ depend on~$\theta$.

Instead of directly appealing to Conjecture~\ref{conjSCS}, we will instead need a corollary of Conjecture~\ref{conjSCS} for arithmetic progressions. For $q$ odd and squarefree, let $\chi_q$ be the real Dirichlet character $\left(\frac{\cdot}{q}\right)$. Following \cite[7., p.\ 924-925]{erratum} we will prove:

\begin{corollary}\label{conjSCSAP}
Assume Conjecture \ref{conjSCS}. Then there exist absolute constants $\delta>0$ and $C>0$ such that for all odd squarefree integers $q>1$, all integers $N<q^{\frac{1}{8}}$, all integers $M$, $l$ and $k$ satisfying $q \nmid k$ we have
$$
\left| \sum_{\substack{M \leq n \leq M + N \\ n \equiv l \bmod k}} \chi_q(n) \right| \leq C q^{\frac{1}{8}-\delta}.
$$
\end{corollary}

\begin{proof}
Write $n = km + l$. Then we have
\[
\chi_q(n) = \chi_{(q, k)}(l) \chi_{q/(q, k)}(k) \chi_{q/(q, k)}(m + r),
\]
where $r$ satisfies $kr \equiv l \bmod q/(q, k)$. It follows that
\[
\left| \sum_{\substack{M \leq n \leq M + N \\ n \equiv l \bmod k}} \chi_q(n) \right| \leq \left| \sum_{M' \leq m \leq M' + \frac{N}{k}} \chi_{q/(q, k)}(m) \right|,
\]
where $M' = (M - l)k^{-1} + r$. By our assumption $q \nmid k$, we see that $q/(q, k)$ is an odd squarefree integer greater than one. Hence $\chi_{q/(q, k)}$ is a non-principal real-valued Dirichlet character. Now apply Conjecture \ref{conjSCS}.
\end{proof}

\subsection{Vinogradov's method, after Friedlander, Iwaniec, Mazur, and Rubin}\label{SumsOverPrimes}
Vinogradov's method \cite{Vino1, Vino2} has been substantially simplified by Vaughan~\cite{Vaughan}, and Friedlander et al. \cite[Section 5, p.\ 717-722]{FIMR} gave a nice generalization to number fields. Morally speaking, power-saving estimates in sums over primes follow from power-saving estimates in linear congruence sums (sums of type I) and general bilinear sums (sums of type II). Precisely, by \cite[Proposition 5.2, p.\ 722]{FIMR} with $\vartheta = \delta/4$ and $\theta = 1/48$, Theorem~\ref{mainThm} is a direct consequence of the following two propositions.
\begin{prop}\label{typeIprop}
Assume Conjecture~\ref{conjSCS} holds with $\delta>0$. Then for all $\epsilon>0$, we have
$$
\sum_{\Norm(\aaa)\leq x,\ \mm|\aaa}s_{\aaa}\ll_{\epsilon} x^{1-\frac{\delta}{4}+\epsilon}
$$
uniformly for all non-zero ideals $\mm$ of $\OO_M$ and all $x\geq 2$.
\end{prop}
\begin{prop}\label{typeIIprop}
For each $\epsilon>0$, there exists a constant $c_{\epsilon}>0$ such that 
$$
\sum_{\Norm(\aaa)\leq M}\sum_{\Norm(\bb)\leq N}v_{\aaa}w_{\bb}s_{\aaa\bb}\ll_{\epsilon} (M+N)^{\frac{1}{48}}(MN)^{\frac{47}{48}+\epsilon}
$$
uniformly for all $M, N\geq 2$ and all sequences of complex numbers $\{v_{\aaa}\}$ and $\{w_{\bb}\}$ satisfying $|v_{\aaa}|, |w_{\aaa}| \leq c_{\epsilon}\Norm(\aaa)^{\epsilon}$.
\end{prop}
Note that Proposition~\ref{typeIIprop} is \textit{unconditional} -- it is only for the sums of type I featuring in Proposition~\ref{typeIprop} that we have to assume Conjecture~\ref{conjSCS}. The proof of Proposition~\ref{typeIIprop} is rather standard at this point; similar results in slightly different settings can be found in \cite{FI1, FIMR, Milovic2, Milovic3, KM2}, among others. The substantially more difficult proof of Proposition~\ref{typeIprop} requires us to make a genuine improvement to the argument of Friedlander et al.\ \cite[Section~6]{FIMR}.

\subsection{A fundamental domain for the action of $\OO_M^{\times}$}\label{sDomain}
In the definition of $s_{\aaa}$ in \eqref{defan}, we chose a generator $\alpha$ for the ideal $\aaa$. As we will see in the proofs of Propositions \ref{typeIprop} and \ref{typeIIprop}, when summing over multiple ideals $\aaa$, it will be useful to work with a compatible set of generators. Here we present a suitable set of such generators, given by a standard fundamental domain for the action of $\OO_M^{\times}$ on $\OO_M$.

Recall that $\OO_M^{\times} = \left\langle \zeta_8\right\rangle\times V$, where $V$ is free of rank $3$. The group $V$ acts on $\OO_M$ by multiplication, i.e., there is an action 
$$
\Psi: V\times\OO_M\rightarrow\OO_M
$$
given by $\Psi(\mu, \alpha) = \mu\alpha$. Up to units of finite order, the orbits of $\Psi$ correspond to ideals in $\OO_M$.

Fix an integral basis for $\OO_M$, say $\eta = \{\eta_1,\ldots, \eta_8\}$. If $\alpha = a_1\eta_1+\cdots+a_8\eta_8\in\OO_M$ with $a_i\in\Z$, we call $a_i$ the coordinates of $\alpha$ in the basis $\eta$. The ideal in $\OO_M$ generated by $\alpha$ is also generated by $\mu\alpha$ for any unit $\mu\in V$. As $V$ is infinite, one can choose $\mu$ so that the coordinates of $\mu\alpha$ in the integral basis $\eta$ are arbitrarily large. The following classical result ensures that one can choose $\mu$ so that the coordinates of $\mu\alpha$ are reasonably small.
\begin{lemma}\label{fundDom}
There exists a subset $\DD$ of $\OO_M$ such that:
\begin{enumerate}
	\item $\DD$ is a fundamental domain for the action $\Psi$, i.e., for all $\alpha\in\OO_M$, there exists a unique $\mu\in V$ such that $\mu\alpha\in\DD$; and
	\item every non-zero ideal $\aaa$ in $\OO_M$ has exactly $8$ generators in $\DD$; if $\alpha$ is one such generator, then all such generators are of the form $\zeta_8^j \alpha$, where $j\in\{1, \ldots, 8\}$; and
	\item there exists a constant $C = C(\eta)>0$ such that for all $\alpha\in\DD$, the coordinates $a_i$ of $\alpha$ in the basis $\eta$ satisfy $|a_i|\leq C\cdot\Norm(\alpha)^{\frac{1}{8}}$.
\end{enumerate}
\end{lemma}
For a proof, see \cite{KM2}, based on \cite[Lemma 1, p.\ 131]{LangANT}. We are now ready to prove Propositions \ref{typeIprop} and \ref{typeIIprop}, thereby proving Theorem~\ref{mainThm}.

\section{Proof of Theorem~\ref{mainThm}}\label{SectionMainThm}
As mentioned in Section~\ref{SumsOverPrimes}, thanks to \cite[Proposition 5.2, p.\ 722]{FIMR}, Theorem~\ref{mainThm} reduces to proving the appropriate estimates for sums of type I and sums of type II. 

\subsection{Sums of type I}\label{PROOFtypeIprop}
In this section, we prove Proposition~\ref{typeIprop}. Define $F$ as in \eqref{defF}. We recall that we fixed a rank $3$ subgroup $V$ of $\OO_M$ and a set of representatives $\mu_1, \ldots, \mu_{8}$ for $V/V^2$. Let $\mathfrak{m}$ be an ideal of $\OO_M$ coprime with $F$. Recall the definition of $s_{\aaa}$ in \eqref{defan}. After using Lemma~\ref{fundDom} to transform a sum over ideals in $\OO_M$ to a sum over elements in the fundamental domain $\DD$, our goal becomes to bound the following sum
\[
A(x) = \frac{1}{64} \sum_{\substack{\Norm(\mathfrak{a}) \leq x \\ (\mathfrak{a}, F) = 1,\ \mathfrak{m} \mid \mathfrak{a}}} \sum_{i = 1}^{8}\sum_{j = 1}^{8} [\mu_i\zeta_8^j \alpha] =  \frac{1}{64} \sum_{i = 1}^{8} \sum_{\substack{\alpha \in \DD; \Norm(\alpha) \leq x \\ (\alpha, F) = 1, \mathfrak{m} \mid \alpha}} [\mu_i \alpha],
\]
where, for convenience of notation, we have set $[\beta] = \psi(\beta \bmod F)[\beta]_r$ for $\beta\in\OO_M$. The rough strategy of our proof will be the same as the strategy in \cite[Section~6]{FIMR}, although we will have to make the appropriate adjustments in numerous places. We can simplify several steps thanks to the special properties of the field $M$ as described in Section~\ref{GovField}. At some point, however, the strategy of \cite[Section~6]{FIMR} will no longer suffice, and we will need a new ingredient.

By making changes of variables $\alpha\mapsto \mu_i^{-1}\alpha$, we rewrite the sum above as
\[
A(x) = \frac{1}{64} \sum_{i = 1}^{8} \sum_{\substack{\alpha \in \mu_i \DD; \Norm(\alpha) \leq x \\ (\alpha, F) = 1, \mathfrak{m} \mid \alpha}} [\alpha]
\]
and after splitting the sum into congruence classes modulo $F$, we get
\[
A(x) = \frac{1}{64} \sum_{i = 1}^{8} \sum_{\substack{\rho \bmod F; \\ (\rho, F) = 1}}  \psi(\rho) A(x; \rho, \mu_i),
\]
where
\[
A(x; \rho, \mu_i) =\sum_{\substack{\alpha \in \mu_i \DD; \Norm(\alpha) \leq x \\ \alpha \equiv \rho \bmod F \\ \alpha \equiv 0 \bmod \mathfrak{m}}} [\alpha]_r.
\]
Our goal is to estimate $A(x; \rho, \mu_i)$ for each congruence class $\rho \bmod F$, $(\rho, F) = 1$ and unit $\mu_i$. As a $\Z$-module, the ring $\OO_M$ decomposes as $\OO_M = \Z \oplus \MM$, where $\MM$ is a free $\mathbb{Z}$-module of rank $7$, so that we can write
\[
\MM = \omega_2 \Z + \ldots + \omega_8 \Z
\]
for some $\omega_2, \ldots, \omega_8 \in \OO_M$. This means that $\alpha$ can be written uniquely as
\[
\alpha = a + \beta, \text{ with } a \in \Z, \beta \in \MM,
\]
so the four summation conditions above are equivalent to
\[
a + \beta \in \mu_i\DD, \quad \Norm(a + \beta) \leq x, \quad a + \beta \equiv \rho \bmod F, \quad a + \beta \equiv 0 \bmod \mathfrak{m}.
\]
Part 3 of Lemma \ref{fundDom} implies that the conjugates of $\beta$, say $\beta^{(i)}$ for $1\leq i\leq 8$, satisfy $|\beta^{(i)}| \ll x^{\frac{1}{8}}$ for any embedding $M\hookrightarrow \mathbb{C}$. Because our field $M$ and the integral basis $\{1, \omega_2, \ldots, \omega_8\}$ is fixed, the implied constant is absolute.

Perhaps the main step of \cite[Section~6]{FIMR} is a trick on page 725, which we use to rewrite $[\alpha]_r = \left(\frac{r(\alpha)}{\alpha}\right)$ as
\[
\left(\frac{r(\alpha)}{\alpha}\right) = \left(\frac{r(a + \beta)}{a + \beta}\right) = \left(\frac{r(\beta) - \beta}{a + \beta}\right).
\]
Morally speaking, this allows us to fix $\beta$ and vary $a$, thereby creating a genuine character sum in which the variable of summation does not depend on the conductor of the character. If $\beta = r(\beta)$, then $\beta$ does not contribute to the sum. So we can and will assume $\beta \neq r(\beta)$. By property (P1) in Section~\ref{GovField}, we can write
\[
r(\beta) - \beta = \eta^2 c_0 c
\]
with $c_0, c, \eta\in\OO_M$, $c_0 \mid F$ squarefree, $\eta \mid F^\infty$, and $(c, F) = 1$. Then
\[
\left(\frac{r(\beta) - \beta}{a + \beta}\right) = \left(\frac{\eta^2 c_0 c}{a + \beta}\right) = \left(\frac{c_0 c}{a + \beta}\right) = \left(\frac{c_0}{a + \beta}\right) \left(\frac{c}{a + \beta}\right)
\]
By Lemma \ref{mod8}, the factor $\left(\frac{c_0}{a + \beta}\right)$ depends only on the congruence class of $a + \beta$ modulo~$8 c_0$, and, as $c_0$ is squarefree and divides $F$, it depends only on $\rho$.

Next we claim that
\[
\left(\frac{c}{a + \beta}\right) = \ve_1 \cdot \left(\frac{a + \beta}{c}\right),
\]
where $\ve_1\in\{\pm 1\}$ depends only on $\rho$ and $\beta$. Indeed, $\rho$ determines the congruence class of $a + \beta$ modulo $8$ and $c$ depends only on $\beta$, so an application of Lemma \ref{tQR} proves the claim. Combining everything gives
\[
\left(\frac{r(\alpha)}{\alpha}\right) = \ve_2 \cdot \left(\frac{a + \beta}{c}\right),
\]
where $\ve_2 = \ve_2(\rho, \beta)\in\{\pm 1\}$ depends only on $\rho$ and $\beta$. Having rewritten $\left(\frac{r(\alpha)}{\alpha}\right)$ in a desirable form, we can now split $A(x; \rho, \mu_i)$ as follows
\begin{align*}
A(x; \rho, \mu_i) 
&= \sum_{\substack{\alpha \in \mu_i \DD; \Norm(\alpha) \leq x \\ \alpha \equiv \rho \bmod F \\ \alpha \equiv 0 \bmod \mathfrak{m}}} \left( \frac{r(\alpha)}{\alpha} \right)
= \sum_{\substack{a + \beta \in \mu_i \DD; \Norm(a + \beta) \leq x \\ a + \beta \equiv \rho \bmod F \\ a + \beta \equiv 0 \bmod \mathfrak{m}}} \left( \frac{r(a + \beta)}{a + \beta} \right) \\
&= \sum_{\beta \in \MM} \sum_{\substack{a \in \Z; \\ a + \beta \in \mu_i \DD, \Norm(a + \beta) \leq x \\ a + \beta \equiv \rho \bmod F \\ a + \beta \equiv 0 \bmod \mathfrak{m}}} \left( \frac{r(a + \beta)}{a + \beta} \right)
= \sum_{\beta \in \MM} \sum_{\substack{a \in \Z; \\ a + \beta \in \mu_i \DD; \Norm(a + \beta) \leq x \\ a + \beta \equiv \rho \bmod F \\ a + \beta \equiv 0 \bmod \mathfrak{m}}} \ve_2(\rho, \beta) \left(\frac{a + \beta}{c}\right) \\
&\leq \sum_{\beta \in \MM} |T(x; \beta, \rho, \mu_i)|,
\end{align*}
where $T(x; \beta, \rho, \mu_i)$ is defined as
\[
T(x; \beta, \rho, \mu_i) = \sum_{\substack{a \in \Z; \\ a + \beta \in \mu_i \DD, \Norm(a + \beta) \leq x \\ a + \beta \equiv \rho \bmod F \\ a + \beta \equiv 0 \bmod \mathfrak{m}}} \left(\frac{a + \beta}{c}\right).
\]
From now on we treat $\beta$ as fixed and estimate $T(x; \beta, \rho, \mu_i)$. Recall that $c$ is odd and hence no ramified prime can divide the ideal $(c) = c\OO_M$ by property (P2) in Section \ref{GovField}. This implies that $(c)$ can be factored as
\[
(c) = \mathfrak{g} \mathfrak{q},
\]
where, similarly as in \cite[(6.21), p.\ 727]{FIMR}, $\mathfrak{g}$ consists of all prime ideals dividing $(c)$ that are of degree greater than one or unramified primes of degree one for which some conjugate is also a factor of $(c)$. By construction $\mathfrak{q}$ consists of all the remaining primes dividing $c\OO_M$. Then $q := N\mathfrak{q}$ is a square-free integer and $g := N\mathfrak{g}$ is a squarefull number coprime with $q$. There exists a rational integer $b$ with $b \equiv \beta \bmod \mathfrak{q}$ by an application of the Chinese remainder theorem. Again, as $c$ depends on $\beta$ and not on $a$, so also $b$ is a rational integer that depends on $\beta$ and not on $a$. We get
\[
\left(\frac{a + \beta}{c}\right) = \left(\frac{a + \beta}{\mathfrak{g}}\right) \left(\frac{a + \beta}{\mathfrak{q}}\right) = \left(\frac{a + \beta}{\mathfrak{g}}\right) \left(\frac{a + b}{\mathfrak{q}}\right).
\] 
Define $g_0$ as the radical of $g$, i.e.,
\[
g_0 = \prod_{p \mid g} p.
\]
Note that the quadratic residue symbol $\left(\frac{\alpha}{\mathfrak{g}}\right)$ is periodic in $\alpha$ modulo $\mathfrak{g}^\ast = \prod_{\mathfrak{p} \mid \mathfrak{g}} \mathfrak{p}$. Since $\mathfrak{g}^\ast$ divides $g_0$, we conclude that the symbol $\left(\frac{a + \beta}{\mathfrak{g}}\right)$ is periodic of period $g_0$ as a function of $a\in\Z$. We split $T(x; \beta, \rho, \mu_i)$ into congruence classes modulo $g_0$, giving
\begin{equation}\label{eFin}
|T(x; \beta, \rho, \mu_i)| \leq \sum_{a_0 \bmod g_0} |T(x; \beta, \rho, \mu_i, a_0)|,
\end{equation}
where
$$
T(x; \beta, \rho, \mu_i, a_0) = \sum_{\substack{a \in \Z; \\ a + \beta \in \mu_i \DD, \Norm(a + \beta) \leq x \\ a + \beta \equiv \rho \bmod F \\ a + \beta \equiv 0 \bmod \mathfrak{m} \\ a \equiv a_0 \bmod g_0}} \left(\frac{a + b}{\mathfrak{q}}\right).
$$
Note that $a + \beta \in \mu_i \DD$ implies that $a \ll x^{\frac{1}{8}}$, where the implied constant depends only on one of the eight units $\mu_i$. The condition $\Norm(a + \beta) \leq x$ for fixed $\beta$ and $x$ is a polynomial inequality of degree $8$ in $a$. So the summation variable $a\in\Z$ runs over a collection of at most $8$ intervals whose endpoints depend on $\beta$ and $x$. But from $a \ll x^{1/8}$ we see that for the length $L$ of each such interval we have $L \ll x^{1/8}$.

Furthermore, the congruences $a + \beta \equiv \rho \bmod F$, $a + \beta \equiv 0 \bmod \mathfrak{m}$ and $a \equiv a_0 \bmod g_0$ mean that $a$ runs over a certain arithmetic progression of modulus $k$ which divides $g_0mF$, where $m := \Norm\mathfrak{m}$. Hence, we see that the inner sum in \eqref{eFin} can be rewritten as at most $8$ sums, each of which runs over an arithmetic progression of modulus $k$ in a single segment of length $\ll x^{1/8}$.

As $q = \Norm(\mathfrak{q})$ is squarefree, $\left(\frac{\cdot}{\mathfrak{q}}\right)$ is the real primitive Dirichlet character of modulus $q$, and hence we have at most $8$ incomplete character sums of length $\ll x^{\frac{1}{8}}$ and modulus $q \ll x$. When the modulus $q$ of the Dirichlet character divides the modulus $k$ of the arithmetic progression, one can not expect to get cancellation. For now we assume that $q \nmid k$, and we will deal with the case $q \mid k$ later on. Corollary~\ref{conjSCSAP} implies that
$$
T(x; \beta, \rho, \mu_i, a_0) \ll x^{\frac{1}{8} - \delta},
$$
and hence that
\begin{align}\label{eCon}
T(x; \beta, \rho, \mu_i) \ll g_0x^{\frac{1}{8} - \delta}.
\end{align}
Just as in \cite{FIMR}, the implied constant above does not depend on $\beta$ because Conjecture~\ref{conjSCS}, and so also Corollary~\ref{conjSCSAP}, encompasses all incomplete character sums of length $\ll x^{\frac{1}{8}}$, regardless of the endpoints of the interval being summed over.

We still need to deal with the case $q \mid k$. Certainly, this implies $q \mid m$. So (\ref{eCon}) holds if $q \nmid m$. Hence, by the definition of $(c)$ and the factorization $(c) = \mathfrak{g} \mathfrak{q}$, we have (\ref{eCon}) unless
\begin{align}
\label{eExp}
p \mid \Norm(\alpha - r(\alpha)) \implies p^2 \mid mF\Norm(\alpha - r(\alpha)).
\end{align}
We write $A_\square(x; \rho, \mu_i)$ for the contribution to $A(x; \rho, \mu_i)$ with (\ref{eExp}). We have
\[
A_\square(x; \rho, \mu_i) \leq |\{\alpha \in \mu_i \DD : N\alpha \leq x, \ p \mid \Norm(\alpha - r(\alpha)) \implies p^2 \mid mF\Norm(\alpha - r(\alpha))\}|.
\]
Decompose $\OO_M$ as
\[
\OO_M = \Z[\sqrt{-2}] \oplus \MM',
\]
where $\MM'$ is a free $\mathbb{Z}$-module of rank $6$. Then we get an injective map $\MM' \rightarrow \OO_M$ given by $\alpha \mapsto \alpha - r(\alpha)$. Since $\alpha \in \mu_i \DD$ and $\Norm(\alpha) \leq x$, we know that all the conjugates $|\alpha^{(k)}|$ are $\ll x^{1/8}$. If we write
\[
\alpha = a + b\sqrt{-2} + m'
\]
with $a, b \in \Z$ and $m' \in \MM'$, then it follows that $|a|, |b| \leq y$ and furthermore all the conjugates of $\gamma = \alpha - r(\alpha)$ satisfy $|\gamma^{(k)}| \leq y$ for some $y \asymp x^{\frac{1}{8}}$. Therefore, we have
\[
A_\square(x; \rho, \mu_i) \leq y^2 |\{\gamma \in \OO_M: |\gamma^{(k)}| \leq y, \ p \mid \Norm(\gamma) \implies p^2 \mid mF\Norm(\gamma)\}|.
\]
Since it is easier to count ideals than integers, we replace $\gamma$ by the principal ideal it generates. We remark that an ideal $\mathfrak{b}$ with $\Norm\mathfrak{b} \leq y^8$ has $\ll (\log y)^8$ generators satisfying $|\gamma^{(k)}| \leq y$ for all $k$. Hence
\[
A_\square(x; \rho, \mu_i) \ll x^{\frac{1}{4}} (\log x)^8 |\{\mathfrak{b} \subseteq \OO_M : \Norm\mathfrak{b} \leq y^8, \ p \mid \Norm\mathfrak{b} \implies p^2 \mid mF\Norm\mathfrak{b}\}|.
\]
Now we can use the multiplicative structure of the ideals in $\OO_M$, giving the bound
\[
A_\square(x; \rho, \mu_i) \ll x^{\frac{1}{4}} (\log x)^8 \sum_{\substack{b \leq y^8 \\ p \mid b \implies p^2 \mid mFb}} \tau(b),
\]
where $b$ runs over the positive rational integers and $\tau(b)$ counts the number of ideals in $M$ with norm $b$. Then we have $\tau(b) \ll b^\epsilon$. Note that we can assume $m \leq x$ because otherwise $A(x)$ is the empty sum. Hence, recalling that $y \asymp x^{\frac{1}{8}}$, we conclude that
\[
A_\square(x; \rho, \mu_i) \ll x^{\frac{3}{4} + \epsilon},
\]
where the implied constant depends only on $\epsilon$.

Define $A_0(x; \rho, \mu_i)$ to be the contribution of $A(x; \rho, \mu_i)$ of the terms $\alpha = a + \beta$ not satisfying (\ref{eExp}). We have
\[
A(x; \rho, \mu_i) = A_\square(x; \rho, \mu_i) + A_0(x; \rho, \mu_i).
\]
To estimate $A_0(x; \rho, \mu_i)$ we can use (\ref{eCon}) for every relevant $\beta$. Unfortunately, the bound (\ref{eCon}) is only good when $g_0$ is small. So we make the further partition
\[
A_0(x; \rho, \mu_i) = A_1(x; \rho, \mu_i) + A_2(x; \rho, \mu_i),
\]
where the components run over $\alpha = a + \beta$ with $\beta$ satisfying
\[
g_0 \leq Z \text{ in the sum } A_1(x; \rho, \mu_i),
\]
\[
g_0 > Z \text{ in the sum } A_2(x; \rho, \mu_i).
\]
Here $Z$ is at our disposal and we choose it later. It is here that we must improve on the bounds of \cite{FIMR}. In their proof they define three sums
\[
g_0 \leq Z \text{ in the sum } A_1(x; \rho, \mu_i),
\]
\[
g_0 > Z, g \leq Y \text{ in the sum } A_2(x; \rho, \mu_i),
\]
\[
g_0 > Z, g > Y \text{ in the sum } A_3(x; \rho, \mu_i),
\]
with $Z \leq Y$ at their disposal. Following the proof in \cite{FIMR} would give
\[
A_0(x; \rho, \mu_i) \ll x^\epsilon (Zx^{1 - \delta} + Y^{-\frac{1}{2}} x^{1 + \frac{1}{4}} + Z^{-1} \log Y x + Y^{\frac{5}{2}} x^{\frac{1}{4}}),
\]
and it is easily seen that there is no choice of $Z \leq Y$ that makes $A_0(x; \rho ,\mu_i) \ll x^{1 - \theta_1}$ for some $\theta_1 > 0$. Our proof is conceptually simpler and provides sharper bounds.

We estimate $A_1(x; \rho, \mu_i)$ as in \cite{FIMR} by using (\ref{eCon}) and summing over $\beta \in \MM$ satisfying $|\beta^{(1)}|, \ldots, |\beta^{(8)}| \ll x^{\frac{1}{8}}$ to obtain
\[
A_1(x; \rho, \mu_i) \ll Zx^{1 - \delta}.
\]
Our next goal is to estimate $A_2(x; \rho , \mu_i)$. We keep the condition $\alpha - r(\alpha) \equiv 0 \bmod \mathfrak{g}$, giving
\begin{align}
\label{eTail1}
|A_2(x; \rho, \mu_i)| \leq y^2 \sum_{\substack{\mathfrak{g} \\ g_0 > Z}} E_\mathfrak{g}(y),
\end{align}
where $y \asymp x^{1/8}$ and
\[
E_\mathfrak{g}(y) := |\{\gamma \in \MM'' : \gamma \equiv 0 \bmod \mathfrak{g}, |\gamma^{(k)}| \leq y \text{ for all } k\}|.
\]
Here $\MM''$ is by definition the image of $\MM'$ under the map $\beta \mapsto \beta - r(\beta)$. Let $\eta_3, \ldots, \eta_8$ be a $\mathbb{Z}$-basis of $\MM''$. We view $\MM'' \subseteq \mathbb{R}^6$ via $a_3 \eta_3 + \ldots + a_8 \eta_8 \mapsto (a_3, \ldots, a_8)$. In this way we identify $\MM''$ with $\mathbb{Z}^6$, so $\MM''$ becomes a lattice in $\mathbb{R}^6$. Furthermore, define $\Lambda_\mathfrak{g}$ as
\[
\Lambda_\mathfrak{g} := \{\gamma \in \MM'' : \gamma \equiv 0 \bmod \mathfrak{g}\}.
\]
Then it is easily seen that $\Lambda_\mathfrak{g}$ is a sublattice of $\MM''$.

We further define
\[
S_x = \{(a_3, \ldots, a_8) \in \mathbb{R}^6 : |a_i| \leq c_1x^{\frac{1}{8}}\},
\]
where the constant $c_1>0$ is taken large enough such that
\begin{align}
\label{eTail2}
E_\mathfrak{g}(y) \leq |S_x \cap \Lambda_\mathfrak{g}|.
\end{align}
Note that $S_x = x^{\frac{1}{8}}S_1$, which implies that $\text{Vol}(S_x) = x^{\frac{3}{4}} \text{Vol}(S_1)$. Because $S_1$ is a $6$-dimensional hypercube, it has $12$ sides. Hence there exist an absolute constant $L$ and functions $\varphi_1, \ldots, \varphi_{12}: [0, 1]^{5} \rightarrow \mathbb{R}^6$ satisfying a Lipschitz condition
\[
|\varphi_i(a) - \varphi_i(b)| \leq L |a - b|
\]
for $a, b \in [0, 1]^5$, $i = 1, \ldots, 12$ such that the boundary of $S_1$, denoted by $\partial S_1$, is covered by the images of the $\varphi_i$. Then $x^{\frac{1}{8}} \varphi_1, \ldots, x^{\frac{1}{8}} \varphi_{12}$ are Lipschitz functions for $\partial S_x = \partial x^{\frac{1}{8}} S_1 = x^{\frac{1}{8}} \partial S_1$. Hence we can choose $x^{\frac{1}{8}}L$ as the Lipschitz constant for $S_x$.

We now apply Theorem 5.4 of \cite{Widmer}, which gives
\begin{align}
\label{eWidmer}
\left| |S_x \cap \Lambda_\mathfrak{g}| - \frac{\text{Vol}(S_x)}{\det \Lambda_\mathfrak{g}}\right| \ll_L \max_{0 \leq i < 6} \frac{x^{\frac{i}{8}}}{\lambda_{\mathfrak{g}, 1} \cdot \ldots \cdot \lambda_{\mathfrak{g}, i}},
\end{align}
where $\lambda_{\mathfrak{g}, 1}, \ldots, \lambda_{\mathfrak{g}, 6}$ are the successive minima of $\Lambda_\mathfrak{g}$ and $\ll_L$ means that the implied constant may depend on $L$. Our next goal is to give a lower bound for $\lambda_{\mathfrak{g}, 1}$. 

So let $\gamma \in \Lambda_\mathfrak{g}$ be non-zero. Then $\mathfrak{g} \mid \gamma$ and hence $g \mid \Norm(\gamma)$. Write $\gamma = (a_3, \ldots, a_8)$. We fix some small $\epsilon > 0$. If $a_3, \ldots, a_8 \leq c_2 g^{\frac{1}{8} - \epsilon}$ for some sufficiently small absolute constant $c_2>0$, we obtain $\Norm(\gamma) < g$. Since $g \mid \Norm(\gamma)$, we conclude that $\Norm(\gamma) = 0$, contradiction. Hence there is an $i$ with $a_i > c_2g^{\frac{1}{8} - \epsilon}$. This implies that the length of $\gamma$ satisfies $||\gamma|| \gg g^{\frac{1}{8} - \epsilon}$ and therefore 
\begin{align}
\label{eM1}
\lambda_{\mathfrak{g}, 1} \gg g^{\frac{1}{8} - \epsilon}
\end{align}
By Minkowski's second theorem and (\ref{eM1}) we find that
\begin{align}
\label{eLD}
\det \Lambda_\mathfrak{g} \gg g^{\frac{3}{4} - 6\epsilon}.
\end{align}
Combining (\ref{eWidmer}), (\ref{eM1}) and (\ref{eLD}) gives
\begin{align}
\label{eTail3}
|S_x \cap \Lambda_\mathfrak{g}| \ll \frac{x^{\frac{3}{4}}}{g^{\frac{3}{4} - 6\epsilon}} + \frac{x^{\frac{5}{8}}}{g^{\frac{5}{8} - 5\epsilon}} \ll \frac{x^{\frac{3}{4}}}{g^{\frac{3}{4} - 6\epsilon}}.
\end{align}
Plugging (\ref{eTail2}) and (\ref{eTail3}) back in (\ref{eTail1}) gives
\[
|A_2(x; \rho, \mu_i)| 
\leq y^2 \sum_{\substack{\mathfrak{g} \\ g_0 > Z}} E_\mathfrak{g}(y) \leq y^2 \sum_{\substack{\mathfrak{g} \\ g_0 > Z}} |S_x \cap \Lambda_\mathfrak{g}| 
\ll \sum_{\substack{\mathfrak{g} \\ g_0 > Z}} \frac{x}{g^{\frac{3}{4} - 6\epsilon}}.
\]
We rewrite the last sum as
\begin{align*}
\sum_{\substack{\mathfrak{g} \\ g_0 > Z}} \frac{x}{g^{\frac{3}{4} - 6\epsilon}} 
&= x \sum_{\substack{g \leq x \\ g \text{ squarefull} \\ g_0 > Z}} \frac{\tau(g)}{g^{\frac{3}{4} - 6\epsilon}}
\ll x^{1 + \epsilon'} \sum_{\substack{g \leq x \\ g \text{ squarefull} \\ g_0 > Z}} \frac{1}{g^{\frac{3}{4} - 6\epsilon}} \\
&= x^{1 + \epsilon'} \sum_{\substack{g \leq x \\ g \text{ squarefull} \\ g_0 > Z}} g^{-\frac{1}{4} +6\epsilon} \frac{1}{g^{\frac{1}{2}}}
\leq x^{1 + \epsilon'} Z^{-\frac{1}{2} + 3\epsilon} \sum_{\substack{g \leq x \\ g \text{ squarefull} \\ g_0 > Z}} \frac{1}{g^\frac{1}{2}} \\
&\leq x^{1 + \epsilon'} Z^{-\frac{1}{2} + 3\epsilon} \sum_{\substack{g \leq x \\ g \text{ squarefull}}} \frac{1}{g^\frac{1}{2}}
\ll x^{1 + \epsilon'} Z^{-\frac{1}{2} + 3\epsilon} \log x.
\end{align*}
By picking $Z = X^{\frac{\delta}{2}}$, $\epsilon$ and $\epsilon'$ sufficiently small, we get the desired result with $\theta_1 = \frac{\delta}{4}$.

\section{Sums of type II}\label{PROOFtypeIIprop}
Our goal in this section is to prove Proposition~\ref{typeIIprop}, thereby completing the proof of Theorem~\ref{mainThm}. A power-saving bound for the bilinear sum in Proposition~\ref{typeIIprop} is possible because the symbol
$$
[\alpha]_r = \left(\frac{r(\alpha)}{\alpha}\right)
$$
is \textit{not} multiplicative in $\alpha$ but instead satisfies the following elegant identity, analogous to \cite[(3.8), p.\ 708]{FIMR}. Let $\alpha$ and $\beta$ be odd elements in $\OO_M$. Then
\begin{equation}\label{twistMult}
[\alpha\beta]_r = \left(\frac{r(\alpha\beta)}{\alpha\beta}\right) = [\alpha]_r[\beta]_r\left(\frac{r(\alpha)}{\beta}\right)\left(\frac{r(\beta)}{\alpha}\right) = \ve_3\cdot [\alpha]_r[\beta]_r\gamma(\alpha, \beta),
\end{equation}
where
\begin{equation}\label{defGamma}
\gamma(\alpha, \beta) = \left(\frac{\beta}{r(\alpha)r^3(\alpha)}\right),
\end{equation}
and $\ve_3\in\{\pm 1\}$ depends only on the congruence classes of $\alpha$ and $\beta$ modulo $8$ (see Lemma~\ref{tQR}). We remark here that the natural one-line proof of \eqref{twistMult} should be contrasted with the rather involved proofs of \cite[Lemma 20.1, p.\ 1021]{FI1} and \cite[Proposition 8, p.\ 31]{Milovic2}. It would be very interesting to find a common source of these identities, if it exists.

With $\mu_1, \ldots, \mu_{8}$ and $[\cdot] = \psi(\cdot \bmod F)[\cdot]_r$ is as in the beginning of Section~\ref{PROOFtypeIprop}, we see that the bilinear sum from Proposition~\ref{typeIIprop} is equal to
$$
\frac{1}{64}\sum_{\zeta\in\left\langle \zeta_8\right\rangle}\sum_{i = 1}^{8}B(M, N; \zeta, i)
$$
where
\begin{equation}\label{sumBZI}
B(M, N; \zeta, i) = \sum_{\substack{\alpha\in\DD(M)}}\sum_{\substack{\beta\in\DD(N)}}v_{\alpha}w_{\beta}[\zeta \mu_i \alpha \beta].
\end{equation}
Here $\DD(X) = \{x\in\DD:\ \Norm(x)\leq X\}$; $v_{\alpha}$ (resp.\ $w_{\beta}$) depends only on the ideal generated by $\alpha$ (resp.\ $\beta$); and, the double sum over $\alpha$ and $\beta$ is assumed to be supported on $\alpha$ and $\beta$ such that $(\alpha\beta, F) = 1$.

The condition $(\alpha\beta, F) = 1$ is equivalent to the two conditions $(\alpha, F) = 1$ and $(\beta, F) = 1$. Hence we can decompose the sum \eqref{sumBZI} into $(\#(\OO_M/F\OO_M)^{\times})^2$ sums $B(M, N; \zeta, i, \rho_1, \rho_2)$ where we further restrict the support of $\alpha$ and $\beta$ to fixed invertible congruence classes modulo~$F$, i.e.,
\begin{equation}\label{supportwz}
\alpha\equiv \rho_1\bmod F\ \ \ \ \text{ and }\ \ \ \ \beta\equiv \rho_2\bmod F.
\end{equation}
Hence, with $\ve_4 = \psi(\zeta \mu_i \rho_1\rho_2\bmod F)$ fixed for fixed $\zeta$, $\mu_i$, $\rho_1$, and $\rho_2$, we have
\begin{equation}\label{sumBZIR}
B(M, N; \zeta, i, \rho_1, \rho_2) = \ve_4 \sum_{\substack{\alpha\in\DD(M)}}\sum_{\substack{\beta\in\DD(N)}}v_{\alpha}w_{\beta}[\zeta \mu_i \alpha \beta]_r,
\end{equation}
where we again note that the support of $\alpha$ and $\beta$ is restricted to \eqref{supportwz}. To prove Proposition~\ref{typeIIprop}, it suffices to prove the desired estimate for each of the
$$
64\cdot \left(\#(\OO_M/F\OO_M)^{\times}\right)^2
$$
sums $B(M, N; \zeta, i, \rho_1, \rho_2)$. To this end, we now take advantage of the special non-multiplicativity of the spin symbol $[\cdot]_r$. By \eqref{twistMult}, we can unfold $[\zeta \mu_i \alpha\beta]_r$ into the product
$$
[\zeta \mu_i \alpha\beta]_r = \ve_5[\alpha\beta]_r[\zeta \mu_i]_r\gamma(\zeta \mu_i, \alpha\beta).
$$
The factor $\ve_5\in\{\pm 1\}$ depends only on the congruence classes $\zeta \mu_i\bmod 8$ and $\alpha\beta\bmod 8$, the factor $[\zeta \mu_i]_r$ does not depend on $\alpha$ and $\beta$ in any way, and the factor
$$
\gamma(\zeta \mu_i, \alpha\beta) = \left(\frac{\zeta \mu_i}{r(\alpha\beta)r^3(\alpha\beta)}\right)
$$
is determined by the congruence class $r(\alpha\beta)r^3(\alpha\beta)\bmod 8$, by Lemma~\ref{mod8}. As $8$ divides $F$, all of these congruence classes are determined by $\zeta$, $\mu_i$, $\rho_1$ and $\rho_2$. Hence
\begin{equation}\label{sumBZIR2}
B(M, N; \zeta, i, \rho_1, \rho_2) = \ve_6\cdot \sum_{\substack{\alpha\in\DD(M)}}\sum_{\substack{\beta\in\DD(N)}}v_{\alpha}w_{\beta}[\alpha\beta]_r,
\end{equation}
where $\ve_6= \ve_6(\zeta, \mu_i, \rho_1, \rho_2)$ depends only on $\zeta$, $\mu_i$, $\rho_1$, and $\rho_2$ but not on $\alpha$ and $\beta$. Next, using \eqref{twistMult} again, we get
\begin{equation}\label{sumBZIR3}
B(M, N; \zeta, i, \rho_1, \rho_2) = \ve_7\cdot \sum_{\substack{\alpha\in\DD(M)}}\sum_{\substack{\beta\in\DD(N)}}v_{\alpha}'w_{\beta}'\gamma(\alpha, \beta),
\end{equation}
where $\ve_7$ depends only on $\zeta$, $\mu_i$, $\rho_1$, and $\rho_2$, and 
$$
v_{\alpha}' = v_{\alpha}\cdot [\alpha]_r\ \ \ \ \ \text{ and }\ \ \ \ \ \ w_{\beta}' = w_{\beta}\cdot [\beta]_r.
$$
The sum in \eqref{sumBZIR3} has exactly the same shape as \cite[(3.2), p.\ 11]{KM2}. Moreover, the function $\gamma$ satisfies the properties (P1)-(P3) on page 11 of \cite{KM2}; indeed, (P1) follows by Lemma~\ref{tQR}, and (P2) is clear. For (P3), suppose that $r(\alpha)r^3(\alpha)\OO_M = \aaa^2$ for some odd ideal $\aaa\subset \OO_M$. Then, as $r(\alpha)r^3(\alpha)$ is fixed by $r^2$ and is thus an odd element of $\Q(\zeta_8)$, we have
$$
r(\alpha)r^3(\alpha)\Z[\zeta_8] = \aaa'^2
$$
for some odd ideal $\aaa'\subset \Z[\zeta_8]$. Taking norms to $\Q$, we get that
$$
\Norm_{M/\Q}(\alpha) = \Norm_{M/\Q}(r(\alpha)) = \Norm_{\Q(\zeta_8)/\Q}(r(\alpha)r^3(\alpha))= \Norm_{\Q(\zeta_8)/\Q}(\aaa')^2.
$$
Hence if $\Norm_{M/\Q}(\alpha)$ is not a square, we see that $r(\alpha)r^3(\alpha)$ does not generate the square of an ideal in $\OO_M$, and so
$$
\sum_{\xi\bmod \Norm(\alpha)\OO_M}\gamma(\alpha, \xi) = \Norm(\alpha)^6\cdot \sum_{\xi\bmod r(\alpha)r^3(\alpha)}\left(\frac{\xi}{r(\alpha)r^3(\alpha)}\right) = \Norm(\alpha)^6\cdot 0 = 0,
$$
which proves (P3). Proposition~\ref{typeIIprop} now follows by \cite[Proposition 3.6, p.\ 11]{KM2}.

\section{Proof of Theorem~\ref{Thmi}}\label{SectionThmi}
We will now deduce Theorem~\ref{Thmi} from Theorem~\ref{mainThm} by choosing the factor $\psi$ in the definition of $s_{\aaa}$ appropriately. First note that Theorem~\ref{Thmi} is equivalent to the statement that
$$
\sum_{p\leq X}a_p \ll X^{1-\delta'},
$$
where 
\begin{equation}\label{defap}
a_p = 
\begin{cases}
1 & \text{if }h(-4p)\equiv 0\bmod 16 \\
-1 & \text{if }h(-4p)\equiv 8\bmod 16 \\
0 & \text{otherwise.}
\end{cases}
\end{equation}
We will use an algebraic criterion for the $16$-rank due to Bruin and Hemenway~\cite{BH}. Let $p$ be a prime number such that $h(-4p) \equiv 0\bmod 8$, i.e., such that $p$ splits completely in $M/\Q$. As in Section~\ref{GovField}, set $K_1 = \Q(i, \sqrt{1+i})$. Let $\rho$ be a prime in $\OO_{K_1}$ dividing $p$, and let $\delta_p$ be an element of $\OO_{K_1}$ such that $\Norm_{K_1/\Q(i)}(\delta_p) = p$ and such that $\delta_p\notin \rho\OO_{K_1}$. Bruin and Hemenway proved that
\begin{equation}\label{crit0}
h(-4p)\equiv 0\bmod 16\Longleftrightarrow \left(\frac{\delta_p\cdot \sqrt{1+i}}{\rho}\right)_{K_1} = 1.
\end{equation}
We will now interpret this symbol as a quadratic residue symbol in $M$. Recall the definition of $r$ and $s$ and the field diagram in Section~\ref{GovField}.

Let $\pi$ be a prime in $\OO_M$ dividing $p$ such that
\begin{equation}\label{defw}
p = \prod_{\sigma\in\Gal(M/\Q)}\sigma(\pi).
\end{equation}
We define elements $\rho$ and $\delta_p$ in $\OO_{K_1}$ by setting $\rho = \pi\cdot s(\pi)$ and
$$
\delta_p = r(\pi)r^2(\pi)\cdot sr(\pi)sr^2(\pi).
$$
Note that $\Norm_{K_1/\Q(i)}(\delta_p) = \delta_p\cdot r^2(\delta_p) = p$ and $\delta_p\notin\rho\OO_{K_1}$, so that $\rho$ and $\delta_p$ satisfy the assumptions implicit in criterion \eqref{crit0}. Next, note that since $p$ splits completely in $M/\Q$, the inclusion $\OO_{K_1}\hookrightarrow \OO_M$ induces an isomorphism of finite fields of order $p$
$$
\OO_{K_1}/\rho\OO_{K_1}\cong \OO_M/\pi\OO_M.
$$
Hence
$$
\left(\frac{\delta_p\cdot \sqrt{1+i}}{\rho}\right)_{K_1} = \left(\frac{\delta_p\cdot \sqrt{1+i}}{\pi}\right)_{M},
$$
and so
\begin{equation}\label{crit1}
h(-4p)\equiv 0\bmod 16 \Longleftrightarrow \left(\frac{r(\pi)r^2(\pi)\cdot sr(\pi)sr^2(\pi)\cdot \sqrt{1+i}}{\pi}\right)_{M} = 1.
\end{equation}
The above quadratic residue symbol factors into five quadratic residue symbols, the first four of which are of the form $\left(\frac{\sigma(\pi)}{\pi}\right)_{M}$ with $\sigma$ in $\{r, r^2, sr, sr^2\}$, and the last one of which is $\left(\frac{\sqrt{1+i}}{\pi}\right)_{M}$. For $\sigma\in\Gal(M/\Q)$, we set
$$
[\alpha]_{\sigma} = \left(\frac{\sigma(\alpha)}{\alpha}\right)_M.
$$
We will now show that when $\sigma$ is an element of order $2$, the spin symbol $[\alpha]_{\sigma}$ can be absorbed into the factor $\psi$. One part of what follows is an adaptation of the treatment of such spins in \cite[Section~12, p.\ 745-749]{FIMR}.

\begin{prop}\label{modF}
Let $\alpha \in \OO_M$ be such that $(\alpha, F) = 1$, and let $\sigma$ be an element of order $2$ in $\Gal(M/\Q)$ such that $(\alpha, \sigma(\alpha)) = 1$. Then $[\alpha]_{\sigma}$ depends only on $\sigma$ and on the congruence class of $\alpha$ modulo $F$.
\end{prop}
The proof of our claim proceeds in two steps. The first step will be to reduce to the case $\alpha \equiv 1 \bmod 8$. The second step will be to use the ideas from Section 12 of \cite{FIMR}. Recall the definitions of $\RRR$ and $F$ in Section~\ref{GovField}.
\begin{proof}
As $(\alpha, F) = 1$, we also have $(\alpha, \Delta_M) = 1$. Let $\rho'\in\RRR$ be such that $\alpha \rho' \equiv 1 \bmod \Delta_M$ and in particular, by property (P2) from the beginning of Section~\ref{GovField}, such that $\alpha \rho' \equiv 1 \bmod 8$. We emphasize two important facts. First, note that $\rho'$ depends only on $\alpha\bmod \Delta_M$ and hence only on $\alpha\bmod F$. Second, as $\Norm(\rho')$ divides $F$ and $(\rho')$ is a prime of degree $1$, we have
$$
(\sigma(\rho'), \alpha) = (\sigma(\alpha), \rho') = (\rho', \sigma(\rho')) = 1. 
$$
Hence each of the four factors on the right-hand side of
$$
\left( \frac{\sigma(\alpha \rho')}{\alpha \rho'} \right)_M = \left( \frac{\sigma(\alpha)}{\alpha} \right)_M \left( \frac{\sigma(\rho')}{\alpha} \right)_M \left( \frac{\sigma(\alpha)}{\rho'} \right)_M \left( \frac{\sigma(\rho')}{\rho'} \right)_M
$$
is non-zero. Using Lemma~\ref{tQR} and the assumption that $\sigma$ is an involution, we get
$$
\left( \frac{\sigma(\rho')}{\alpha} \right)_M = \ve_8\cdot \left( \frac{\alpha}{\sigma(\rho')} \right)_M = \ve_8\cdot \left( \frac{\sigma(\alpha)}{\rho'} \right)_M,
$$
where $\ve_8\in\{\pm 1\}$ depends only on $\sigma$ and the congruence classes of $\sigma(\rho')$ and $\alpha$ modulo $8$, both of which depend only on $\sigma$ and $\alpha\bmod F$. Furthermore, $\left( \frac{\sigma(\rho')}{\rho'} \right)_M\in\{\pm 1\}$ also depends only on $\sigma$ and $\alpha\bmod F$. This gives
\begin{equation}\label{reduction1}
\left( \frac{\sigma(\alpha \rho')}{\alpha \rho'} \right)_M = \ve_9\cdot \left( \frac{\sigma(\alpha)}{\alpha} \right)_M \left( \frac{\sigma(\alpha)^2}{\rho'} \right)_M = \ve_9\cdot \left( \frac{\sigma(\alpha)}{\alpha} \right)_M, 
\end{equation}
where $\ve_9\in\{\pm 1\}$ depends only on $\sigma$ and $\alpha\bmod F$. So from now on we may assume that $\alpha \equiv 1 \bmod 8$.

In the interest of not being repetitive, we now refer to the argument used to prove \cite[Proposition 12.1, p.\ 745]{FIMR}. Define $L$ to be the subfield of $M$ fixed by $\left\langle\sigma \right\rangle$. In our case, the discriminant ideal $\Disc(M/L)$ is even, and in fact divides a power of $2\OO_L$. Although the proof of \cite[Proposition 12.1, p.\ 745]{FIMR} relies on $\mathfrak{D}$ being odd in an essential way, we will overcome this by using the fact that $\OO_L$ is a principal ideal domain.

Similarly as in \cite[(12.4), p.\ 747]{FIMR}, one can deduce that
$$
\left( \frac{\sigma(\alpha)}{\alpha} \right)_M = \ve_{10}\left( \frac{-\gamma^2}{\beta} \right)_L,
$$
where $\ve_{10}\in\{\pm 1\}$ depends only on $\sigma$ and $\alpha\bmod 8$, and where $\gamma$ and $\beta$ are defined via
$$
\beta = \frac{1}{2}\left(\alpha+\sigma(\alpha)\right)\equiv 1\bmod 4, \ \ \ \ \ \ \ \ \ \ \ \ \gamma = \frac{1}{2}\left(\alpha-\sigma(\alpha)\right)\equiv 0\bmod 4.
$$
Defining the submodule $\mathcal{M}$ of $\OO_M$ in the same way as on \cite[p.\ 747]{FIMR}, i.e., $\mathcal{M}=\OO_L+\frac{1+\alpha}{2}\OO_L$, we arrive at the identity
$$
\gamma^2\OO_L = \Disc(\mathcal{M}) = \aaa^2\Disc(M/L),
$$
where $\aaa$ is an ideal in $\OO_L$ such that $\OO_M/\mathcal{M}\cong\OO_L/\aaa$. Since $\OO_L$ is a principal ideal domain (see (P1) in Section~\ref{GovField}), we obtain the equation
$$
\gamma^2 = u\cdot a^2\cdot D,
$$
where now $D\in\OO_L$ is some generator of the discriminant $\Disc(M/L)$, $a\in\OO_L$ is some generator of the ideal $\aaa$, and $u\in\OO_L^{\times}$. Then we have
$$
\left( \frac{-\gamma^2}{\beta} \right)_L = \left( \frac{-uD}{\beta} \right)_L,
$$
which, by Lemma~\ref{mod8}, depends only on the congruence class $\beta\bmod 8D$. One can check that $16D$ divides $\Delta_M$ for any involution $\sigma\in\Gal(M/\Q)$, and so $\beta\bmod 8D$ is completely determined by $\sigma$ and the congruence class $\alpha\bmod \Delta_M$. Hence, whenever $\alpha\equiv 1\bmod 8$, the symbol $[w]_{\sigma}$ only depends on $\sigma$ and $\alpha\bmod \Delta_M$. In conjunction with \eqref{reduction1}, this completes the proof of our proposition. 
\end{proof}

If $\rho$ is an invertible class modulo $F$ and $\sigma\in\{r^2, sr, sr^2\}$, we define
$$
\psi_{\sigma}(\rho) = [\alpha]_{\sigma},
$$
where $\alpha$ is any element of $\OO_M$ such that $\alpha\equiv \rho\bmod F$ and such that $(\alpha, \sigma(\alpha)) = 1$; this is well-defined by Proposition~\ref{modF}. Moreover, define
$$
\psi_M(\rho) = \left(\frac{\sqrt{1+i}}{\alpha}\right)_M,
$$
where $\alpha$ is any element of $\OO_M$ such that $\alpha\equiv \rho\bmod F$; this is well-defined by Lemma~\ref{mod8}. We then define
\begin{equation}\label{defpsii}
\psi_0(\rho)=\psi_{r^2}(\rho)\psi_{sr}(\rho)\psi_{sr^2}(\rho)\psi_{M}(\rho).
\end{equation}
We now check that $\psi_0(\alpha \mod F) = \psi_0(\alpha\beta^2 \bmod F)$ for all $\alpha\in\OO_M$ coprime to $F$ and all $\beta\in\OO_M^{\times}$. Indeed, it is clear that $\psi_{M}(\alpha\beta^2\bmod F) = \psi_{M}(\alpha\bmod F)$, and, for any $\sigma\in\Gal(M/\Q)$, we have
\begin{equation}\label{welldef}
\left(\frac{\sigma(\alpha\beta^2)}{\alpha\beta^2}\right) = \left(\frac{\sigma(\alpha\beta^2)}{\alpha}\right) = \left(\frac{\sigma(\alpha)}{\alpha}\right)\left(\frac{\sigma(\beta)^2}{\alpha}\right) = \left(\frac{\sigma(\alpha)}{\alpha}\right).
\end{equation}
From \eqref{crit1}, we now deduce the following criterion for the $16$-rank of $\CL(-4p)$, valid for all but finitely many primes $p$.
\begin{prop}\label{crit2}
Let $p$ be a rational prime such that $p$ splits completely in $M/\Q$ and such that $(p, F) = 1$. Let $\pi$ be any prime in $\OO_M$ dividing $p$. Then
$$
h(-4p)\equiv 0\bmod 16 \Longleftrightarrow \psi_0(\pi\bmod F)\cdot[\pi]_r = 1.
$$
\end{prop}
Let $a_p$ be defined as \eqref{defap}. With $\psi_0$ as in \eqref{defpsii}, we set $\psi = \psi_0$ and define $s_{\aaa}$ as in \eqref{defan}. If $(p, F) = 1$, $p$ splits completely in $M/\Q$, and $\pp$ is any prime ideal in $\OO_M$ lying above~$p$, then Proposition~\ref{crit2} implies that 
\begin{equation}\label{apsp}
a_p = s_{\pp}.
\end{equation}
Since there are only finitely many primes dividing $F$, and since each unramified degree $1$ prime ideal $\pp$ in $\OO_M$ has $8$ conjugates, we have
$$
\sum_{p\leq X}a_p = \sum_{\substack{p\leq X \\ p\nmid F}}a_p + O(1) = \frac{1}{8}\sum_{\substack{\Norm(\pp) = p \leq X \\ p\nmid F}}s_{\pp} + O(1) = \frac{1}{8}\sum_{\substack{\Norm(\pp) = p \leq X}}s_{\pp} + O(1).
$$
The number of prime ideals in $\OO_M$ of degree at least $2$ and of norm $\leq X$ is
$$
\leq 4\sum_{p\leq X^{\frac{1}{2}}}1\ll X^{\frac{1}{2}},
$$
so we have
$$
\sum_{p\leq X}a_p = \frac{1}{8}\sum_{\substack{\Norm(\pp)\leq X}}s_{\pp} + O(X^{\frac{1}{2}}).
$$
Theorem~\ref{mainThm} in conjunction with \eqref{apsp} now gives the desired estimate.

\section{Proof of Theorem~\ref{Thmii}}\label{SectionThmii}
To deduce Theorem~\ref{Thmii} from Theorem~\ref{mainThm}, we will make a different choice for $\psi$. Similarly as in the proof of Theorem~\ref{Thmi}, we define
\begin{equation}\label{defbp}
b_p = 
\begin{cases}
1 & \text{if }h^+(8p)\equiv h(8p)\equiv 0\bmod 8 \\
-1 & \text{if }h^+(8p)+4 \equiv h(8p)\equiv 4\bmod 8 \\
0 & \text{otherwise}
\end{cases}
\end{equation}
and note that Theorem~\ref{Thmii} is equivalent to the estimate
$$
\sum_{p\leq X}b_p \ll X^{1-\delta'}.
$$
Throughout, we fix a primitive $16$th root of unity $\zeta_{16}$ and we set $\zeta_8 = \zeta_{16}^2$, $i = \zeta_8^2$, $\sqrt{-2} = \zeta_8 + \zeta_8^3$, and $\sqrt{2} = \zeta_8 + \zeta_8^{-1}$. As stated in the discussion prior to the statement of Theorem~\ref{Thmii}, for a prime number $p\equiv 1\bmod 4$, we have $h^+(8p) \equiv 0\bmod 8$ if and only if $p$ splits completely in the number field
$$
M' = \Q(\zeta_{16}, \sqrt[4]{2}).
$$
Since $1+i = \zeta_8\sqrt{2}$, we have $M = \Q(\zeta_8, \sqrt{1+i}) = \Q(\zeta_8, \zeta_{16}\sqrt[4]{2})$, and so $M\subset M'$ is a quadratic extension, generated by $\sqrt{\zeta_8}$. We now use a criterion of Kaplan and Williams \cite[p.\ 26]{KW84}. Suppose that $p\equiv 1\bmod 8$, i.e., that $h^+(8p)\equiv 0\bmod 4$. Then we can write
\begin{equation}\label{repnofp}
p = a^2 +b^2 = c^2 + 2d^2,
\end{equation}
with $a, b, c, d\in\Z$. After possibly interchanging $a$ and $b$, we can guarantee that $a$ is odd. Replacing $a$ by $-a$ and $c$ by $-c$ is necessary, we can then ensure that
\begin{equation}\label{critcongcond}
a\equiv c\equiv 1\bmod 4.
\end{equation}
Assume now that $h^+(8p)\equiv 0\bmod 8$, i.e., that $p$ splits completely in $M'/\Q$; this forces the congruence conditions \cite[p.\ 23]{KW84}
$$
a\equiv c\equiv 1\bmod 8, \quad b\equiv 0\bmod 8, \quad d\equiv 0\bmod 4.
$$
With $b_p$ defined as in \eqref{defbp}, and with $\alpha$ and $\beta$ as on page 26 of \cite{KW84}, we have
$$
b_p = \alpha\beta = (-1)^{(a-1+b+2d+h(-4p))/8}.
$$
As $M\subset M'$, it must be that $h(-4p)\equiv 0\bmod 8$, so that with $a_p$ as in the statement of Theorem~\ref{Thmi}, we get
\begin{equation}\label{bpformula}
b_p = (-1)^{(a-1+b+2d)/8}a_p.
\end{equation}
In light of \eqref{apsp}, it remains to express the factor $(-1)^{(a-1+b+2d)/8}$ in terms of a generator $\varpi$ for an ideal in $\OO_M$ lying above $p$. The main difficulty here lies in the sensitivity of the formula \eqref{bpformula} to the conditions \eqref{critcongcond}. Note that
$$
(-1)^{(a-1+b)/8} =
\begin{cases}
1 & \text{if }a+b-1\equiv 0\bmod 16 \\
-1 & \text{if }a+b-1 \equiv 8\bmod 16
\end{cases}
$$
and
$$
(-1)^{d/4} =
\begin{cases}
1 & \text{if }d\equiv 0\bmod 8 \\
-1 & \text{if }d \equiv 4\bmod 8.
\end{cases}
$$
The only units in $\Z[\sqrt{-2}]$ are $\pm 1$, so if $\Norm_{M/\Q(\sqrt{-2})}(\varpi) = c'+d'\sqrt{-2}$, we must have either $(c', d') = (c, d)$ or $(c', d') = (-c, -d)$. Note that $d\equiv 0\bmod 8$ if and only if $-d\equiv 0\bmod 8$, and also $d\equiv 4\bmod 8$ if and only if $-d\equiv 4\bmod 8$. Hence the factor $(-1)^{d/4}$ in \eqref{bpformula} is always equal to $(-1)^{d'/4}$.

The situation for $\Z[i]$ is slightly more complicated. Suppose $\Norm_{M/\Q(i)}(\varpi) = a'+b'i$. Define $e(\varpi)\in\{\pm 1\}$ by the equation
$$
a'+b'\equiv e(\varpi) \bmod 4.
$$
Since $p = a'^2 + b'^2 \equiv 1\bmod 8$, one of $a'$ and $b'$ must be congruent to $0\bmod 4$, and the other is then congruent to $e(\varpi) \bmod 4$. If $e(\varpi) = 1$, then either $(a', b')$ or $(b', a')$ satisfies the same conditions as $(a, b)$ in \eqref{repnofp} and \eqref{critcongcond}, and so $(-1)^{(a-1+b)/8} = (-1)^{(a'+b'-1)/8}$. If $e(\varpi) = -1$, then either $(-a', -b')$ or $(-b', -a')$ satisfies the same conditions as $(a, b)$ in \eqref{repnofp} and \eqref{critcongcond}, and so $(-1)^{(a-1+b)/8} = (-1)^{(-a'-b'-1)/8} = (-1)^{(a'+b'+1)/8}$. In any case, $(-1)^{(a-1+b)/8} = (-1)^{(a'+b'-e(\varpi))/8}$, so that
\begin{equation}\label{bpformula2} 
b_p = (-1)^{(a'+b'-e(\varpi)+2d')/8}a_p.
\end{equation}
Note that the formula \eqref{bpformula2} holds regardless of the congruence classes of $a'$, $b'$, and $d'$. In other words, we have managed to remove the dependence of the formula for $b_p$ on conditions of the shape \eqref{critcongcond}.

Now let $\alpha$ be \textit{any} odd element in $\OO_M$, not necessarily an element of norm $p$. We define $a'', b'', c'', d''\in \Z$ and $e(\alpha)\in\{\pm 1\}$ via the equations
\begin{equation}\label{eqn14}
\Norm_{M/\Q(i)}(\alpha) = a''+b''i, \quad \Norm_{M/\Q(\sqrt{-2})}(\alpha) = c''+d''\sqrt{-2}, \quad a''+b'' = e(\alpha)\bmod 4.
\end{equation}
Let $\rho$ be an invertible congruence class modulo $F$. Define
\begin{equation}\label{defpsit}
\psi_t(\rho) = \frac{1}{2} \left(\exp\left(\frac{\pi i}{8} (a'' + b'' - e(\alpha))\right) + \exp\left(-\frac{\pi i}{8} (a'' + b'' - e(\alpha))\right)\right)\exp\left(\frac{\pi i}{4}d''\right),
\end{equation}
where $\alpha$ is any element of $\OO_M$ such that $\alpha\equiv \rho\bmod F$ and $a''$, $b''$, $d''$, and $e(\alpha)$ are defined via the equations \eqref{eqn14}; this is well-defined since $F$ is divisible by $16$ and $\exp(2\pi i) = 1$. Finally, we define
\begin{equation}\label{defpsi8}
\psi_{M'}(\rho) = \left(\frac{\zeta_8}{\alpha}\right)_M, 
\end{equation}
where $\alpha$ is any element of $\OO_M$ such that $\alpha\equiv \rho\bmod F$; this is well-defined by Lemma~\ref{mod8}.

Suppose $\alpha\in\OO_M$ is coprime to $F$, and suppose $\beta\in\OO_M^{\times}$. Again, it is clear that $\psi_{M'}(\alpha\beta^2\bmod F) = \psi_{M'}(\alpha\bmod F)$. Furthermore, because $\Norm_{M/\Q(i)}(\beta^2) = \Norm_{M/\Q(i)}(\beta)^2 \in \{\pm 1\}$ and $\Norm_{M/\Q(\sqrt{-2})}(\beta^2) = \Norm_{M/\Q(\sqrt{-2})}(\beta)^2 = 1$, and because of the symmetry in \eqref{defpsit} with respect to the transformation $(a'' + b'' - e(\alpha))\mapsto -(a'' + b'' - e(\alpha))$, we also have $\psi_t(\alpha\beta^2\bmod F) = \psi_t(\alpha\bmod F)$.

Finally, with $\psi_0$ defined as in \eqref{defpsii}, we define two functions $\psi_1, \psi_2$ on  $(\OO_M/F\OO_M)^{\times}$ by setting
\begin{equation}\label{defpsiii1}
\psi_1(\rho)=\psi_0(\rho)\psi_t(\rho)
\end{equation}
and
\begin{equation}\label{defpsiii2}
\psi_2(\rho)=\psi_0(\rho)\psi_t(\rho)\psi_{M'}(\rho).
\end{equation}
Now suppose $p$ splits completely in $M/\Q$ and let $\varpi$ be any prime in $\OO_M$ of norm $p$. Since $M' = M(\sqrt{\zeta_8})$, we have
$$
\frac{1}{2}\left(1+ \left(\frac{\zeta_8}{\varpi}\right)_M\right) =
\begin{cases}
1 & \text{if }p\text{ splits completely in }M'/\Q \\
0 & \text{otherwise,}
\end{cases}
$$
so this can be detected by $\psi_{M'}$ for $p$ coprime to $F$. With $a''$, $b''$, and $d''$ defined as in \eqref{eqn14} with $\alpha = \varpi$, we always have $a''+b''-e(\varpi)\equiv 0\bmod 8$; as $\exp(\pi i) = \exp(-\pi i)$, we have 
$$
\psi_t(\varpi) = \exp\left(\frac{\pi i}{8} (a'' + b'' - e(\varpi)+2d'')\right) = (-1)^{(a''+b''-e(\varpi)+2d'')/8}.
$$
Hence from \eqref{bpformula2} and Proposition~\ref{crit2}, supposing also that $(p, F) = 1$, we obtain
\begin{equation}\label{bpformula3}
b_p = \frac{1}{2}\left(\psi_1(\varpi\bmod F) + \psi_2(\varpi\bmod F)\right)[\varpi]_r.
\end{equation}
Now, with $\psi_1$ and $\psi_2$ as in \eqref{defpsiii1} and \eqref{defpsiii2}, respectively, we set $\psi = \psi_1$ (resp.\ $\psi = \psi_2$) and define $s_{1, \aaa}$ (resp.\ $s_{2, \aaa}$) as in \eqref{defan}. If $(p, F) = 1$, $p$ splits completely in $M/\Q$, and $\pp$ is any prime ideal in $\OO_M$ lying above~$p$, then \eqref{bpformula3} implies that 
\begin{equation}\label{bpsp}
b_p = \frac{1}{2}\left(s_{1, \pp} + s_{2, \pp}\right).
\end{equation}
By the same argument as at the end of Section~\ref{SectionThmi}, Theorem~\ref{mainThm} applied to the sequences $\{s_{1, \aaa}\}_{\aaa}$ and $\{s_{2, \aaa}\}_{\aaa}$ proves Theorem~\ref{Thmii}.

\section{Proof of Theorem~\ref{Thmiv}}\label{SectionThmiv}
We start by recalling a criterion due to Bruin and Hemenway \cite[Theorem B, p.\ 66]{BH}. Suppose $p$ is a prime number that splits completely in $M/\Q$ and let $\varpi$ be a prime in $\OO_M$ of absolute norm $p$. Then
$$
\left(\frac{\zeta_8\cdot r(\varpi)r^2(\varpi)sr(\varpi)sr^2(\varpi)\cdot \sqrt{1+i}}{\varpi}\right)_{M} = -1 \Longrightarrow \left(\Z/4\Z\right)^2\hookrightarrow\Sha(E_p) 
$$
(the right hand side implies that $p\in W(3)\setminus W(2)$, where $W(e)$ is defined in \cite[p.\ 65]{BH}; see also \cite[Corollary 2.2, p.\ 67]{BH}). The above product differs from the product in \eqref{crit1} only by the factor $(\frac{\zeta_8}{\varpi})_M$. We thus define $\psi: (\OO_M/F\OO_M)^{\times}\rightarrow\mathbb{C}$ by
$$
\psi(\rho) = \psi_0(\rho)\psi_{M'}(\rho),
$$
where $\psi_0$ is as in \eqref{defpsii} and $\psi_{M'}$ is as in \eqref{defpsi8}. Theorem~\ref{mainThm} applied to the sequence $\{s_{\aaa}\}_{\aaa}$, defined as in \eqref{defan} with $\psi$ as above, now gives the desired result.

\bibliographystyle{plain}
\bibliography{Koymans_Milovic_References}
\end{document}